\documentclass[reqno, 8pt]{amsart}

\usepackage{amsmath,amssymb,amsthm,amsfonts,mathrsfs}
\usepackage{fullpage}
\usepackage{xcolor}
\usepackage{graphicx}
\usepackage{listings}
\usepackage{bbm}

\definecolor{codegreen}{rgb}{0,0.6,0}
\definecolor{codegray}{rgb}{0.5,0.5,0.5}
\definecolor{codepurple}{rgb}{0.58,0,0.82}
\definecolor{backcolour}{rgb}{0.95,0.95,0.92}

\lstdefinestyle{mystyle}{
  backgroundcolor=\color{backcolour},   commentstyle=\color{codegreen},
  keywordstyle=\color{magenta},
  numberstyle=\tiny\color{codegray},
  stringstyle=\color{codepurple},
  basicstyle=\ttfamily\footnotesize,
  breakatwhitespace=false,         
  breaklines=true,                 
  captionpos=b,                    
  keepspaces=true,                 
  numbers=left,                    
  numbersep=5pt,                  
  showspaces=false,                
  showstringspaces=false,
  showtabs=false,                  
  tabsize=2
}

\usepackage{listings}
\usepackage[colorlinks=true,
linkcolor=blue,
anchorcolor=blue,
citecolor=red
]{hyperref}
\allowdisplaybreaks 
\newtheorem{theorem}{Theorem}[section]
\newtheorem{lemma}[theorem]{Lemma}

\newtheorem*{conjecture*}{Conjecture}
\theoremstyle{definition}

\theoremstyle{remark}

\newtheorem*{remark*}{remark}

\renewcommand{\leq}{\leqslant}

\usepackage{colonequals}

\author{\textbf{Runbo Li}}
\address{International Curriculum Center, The High School Affiliated to Renmin University of China, Beijing, China}
\email{runbo.li.carey@gmail.com}

\makeatletter
\@namedef{subjclassname@2020}{\textup{}2020 Mathematics Subject Classification}
\makeatother

\title[]{Primes in almost all short intervals}

\begin{document}
	
\begin{abstract}
The author sharpens a result of Jia (1996), showing that the interval $[n, n+n^{\frac{1}{21.5}+\varepsilon}]$ contains prime numbers for almost all $n$. Watt's mean value bound, a delicate sieve decomposition and more accurate estimates for integrals are used to good effect.
\end{abstract}

\maketitle


\tableofcontents

\section{Introduction}
One of the famous topics in number theory is to find prime numbers in short intervals. In 1937, Cramér \cite{Cramer1937} conjectured that every interval $[n, n+f(n) (\log n)^2]$ contains prime numbers for some $f(n) \to 1$ as $n \to \infty$. The Riemann Hypothesis implies that for all large $n$, the interval $[n, n+n^{\theta}]$ contains $\sim n^{\theta} (\log n)^{-1}$ prime numbers for every $\frac{1}{2}+\varepsilon \leqslant \theta \leqslant 1$. The first unconditional result of this asymptotic formula was proved by Hoheisel \cite{Hoheisel} in 1930 with $\theta=1-\frac{1}{33000}$. After the works of Hoheisel \cite{Hoheisel}, Heilbronn \cite{Heilbronn}, Chudakov \cite{Chudakov}, Ingham \cite{Ingham} and Montgomery \cite{Montgomery}, Huxley \cite{Huxley} proved in 1972 that the above asymptotic formula holds when $\theta>\frac{7}{12}$ by his zero density estimate. In 2024, Guth and Maynard \cite{GuthMaynard} improved this to $\theta>\frac{17}{30}$ by a new zero density estimate.

In 1979, Iwaniec and Jutila \cite{IwaniecJutila} first introduced a sieve method into this problem. They established a lower bound with correct order of magnitude (instead of an asymptotic formula) with $\theta=\frac{13}{23}$. After that breakthrough, many improvements were made and the value of $\theta$ was reduced successively to
$$
\frac{5}{9} \approx 0.5556,\ \frac{11}{20}=0.5500,\ \frac{17}{31} \approx 0.5484,\ \frac{23}{42} \approx 0.5476, 
$$
$$
\frac{1051}{1920} \approx 0.5474,\ \frac{35}{64} \approx 0.5469,\ \frac{6}{11} \approx 0.5455\ \text{ and }\ \frac{7}{13} \approx 0.5385
$$
by Iwaniec and Jutila \cite{IwaniecJutila}, Heath-Brown and Iwaniec \cite{HeathBrownIwaniec}, Pintz \cite{Pintz1} \cite{Pintz2}, Iwaniec and Pintz \cite{IwaniecPintz}, Mozzochi \cite{Mozzochi} and Lou and Yao \cite{LouYao3564} \cite{LouYao} \cite{LouYao2} \cite{LouYao3} respectively. In 1996, Baker and Harman \cite{BakerHarman} presented an alternative approach to this problem. They used the alternative sieve developed by Harman \cite{Harman1} \cite{Harman2} to reduce $\theta$ to $0.535$. Finally, Baker, Harman and Pintz \cite{BHP} further developed this sieve process and combined it with Watt's theorem and showed $\theta=0.525$.

However, if we only consider the prime numbers in "almost all" intervals instead of "all" intervals, the intervals will be much shorter than $n^{0.525}$. In 1943, under the Riemann Hypothesis, Selberg \cite{Selberg1943} showed that Cramér's interval contains primes for almost all $n$ if $f(n) \to \infty$ as $n \to \infty$.  In the same paper, he also showed unconditionally that the interval $[n, n+n^{\frac{19}{77}+\varepsilon}]$ contains prime numbers for almost all $n$. In 1971, Montgomery \cite{Montgomery} improved the exponent $\frac{19}{77}$ to $\frac{1}{5}$ with an asymptotic formula. The zero density estimate of Huxley \cite{Huxley} gives the exponent $\frac{1}{6}$ with an asymptotic formula, and the best asymptotic result now is also due to Guth and Maynard \cite{GuthMaynard}, where they proved the exponent $\frac{2}{15}$.

In 1982, Harman \cite{Harman110} used his alternative sieve method to showed that the interval $[n, n+n^{\frac{1}{10}+\varepsilon}]$ contains prime numbers for almost all $n$. His method can only provide a lower bound instead of an asymptotic formula. The exponent $\frac{1}{10}$ was reduced successively to
$$
\frac{1}{12}=0.0833,\ \frac{14}{159}=0.0881,\ \frac{1}{13}=0.0769,\ \frac{17}{227}=0.0749,\ \frac{1}{13.5}=0.0740,
$$
$$
\frac{1}{14}=0.0714,\ \frac{1}{15}=0.0667,\ \frac{1}{16}=0.0625,\ \frac{1}{18}=0.0556 \text{ and }\ \frac{1}{20}=0.0500
$$
by Harman \cite{Harman1} (and Heath-Brown \cite{Finding}), Lou and Yao \cite{LouYao1985}, Jia \cite{Jia1131} \cite{Jia1132}, Lou and Yao \cite{LouYao1985}, Li \cite{Li227}, Jia \cite{Jia114} (and Watt \cite{Watt114}), Li \cite{Li115}, Baker, Harman and Pintz \cite{BHP2}, Wong \cite{Wong118} (and Jia \cite{Jia118}, Harman [\cite{HarmanBOOK}, Chapter 9]) and Jia \cite{Jia120} respectively. As the strongest result above, Jia \cite{Jia120} used many powerful tools and arguments in his proof, both analytic and combinatorial, and these are extremely complicated. In this paper, we further develop the sieve machinery used in \cite{Jia120} and obtain the following result.
\begin{theorem}\label{t1}
The interval $[n, n+n^{\frac{1}{21.5}+\varepsilon}]$ contains prime numbers for almost all $n$. Specifically, suppose that $B$ is a sufficiently large positive constant, $\varepsilon$ is a sufficiently small positive constant and $X$ is sufficiently large. Then for positive integers $n \in [X, 2X]$, except for $O\left(X (\log X)^{-B} \right)$ values, the interval $[n, n+n^{\frac{1}{21.5}+\varepsilon}]$ contains $\gg  n^{1/21.5+\varepsilon} (\log n)^{-1}$ prime numbers.
\end{theorem}

Throughout this paper, we always suppose that $B$ is a sufficiently large positive constant, $\varepsilon$ is a sufficiently small positive constant, $X$ is sufficiently large and $x \in [X, 2X]$. The letter $p$, with or without subscript, is reserved for prime numbers. Let $c_0$, $c_1$ and $c_2$ denote positive constants which may have different values at different places, and we write $m \sim M$ to mean that $c_1 M<m \leqslant c_2 M$. Let $b=1+\frac{1}{\log X}$, $\delta = \varepsilon^{1/3}$ and $\eta = \frac{1}{2} X^{-\frac{20.5}{21.5}+\varepsilon}$. We use $M(s)$, $N(s)$, $H(s)$ and some other capital letters to denote $1$-bounded Dirichlet polynomials
$$
M(s)=\sum_{m \sim M} a_{m} m^{-s}, \quad N(s)=\sum_{n \sim N} b_{n} n^{-s}, \quad H(s)=\sum_{h \sim H} c_{h} h^{-s},
$$
and we use $L(s)$ to denote a "zeta-factor"
$$
L(s)=\sum_{l \sim L} l^{-s}.
$$

\section{An outline of the proof}
Let $\mathcal{C}$ denote a finite set of positive integers, $p_{j}=X^{t_j}$ in the following sections and put
$$
\mathcal{A}=\{n: x<n \leqslant x+\eta x\}, \quad \mathcal{B}=\{n: x<n \leqslant 2 x\},
$$
$$
\mathcal{C}_d=\{a: a d \in \mathcal{C}\}, \quad P(z)=\prod_{p<z} p, \quad S(\mathcal{C}, z)=\sum_{\substack{a \in \mathcal{C} \\ (a, P(z))=1}} 1.
$$
Then we have
\begin{equation}
\pi(x+\eta x)-\pi(x)=S\left(\mathcal{A},(2 X)^{\frac{1}{2}}\right).
\end{equation}

\textit{Buchstab's identity} is the equation
$$
S\left(\mathcal{C}, z\right) = S\left(\mathcal{C}, w\right) - \sum_{w \leqslant p < z} S\left(\mathcal{C}_{p}, p\right),
$$
where $2 \leqslant w < z$.

In order to prove Theorem~\ref{t1}, we only need to show that $S\left(\mathcal{A},(2 X)^{\frac{1}{2}}\right) >0$. By Buchstab's identity, we have
\begin{align}
\nonumber S\left(\mathcal{A}, (2 X)^{\frac{1}{2}}\right) =&\ S\left(\mathcal{A}, X^{\frac{79}{817}}\right) - \sum_{\frac{79}{817} \leqslant t_1 < \frac{1}{2}} S\left(\mathcal{A}_{p_1}, X^{\frac{79}{817}}\right) \\
\nonumber &+ \sum_{\substack{\frac{79}{817} \leqslant t_1 < \frac{1}{2} \\ \frac{79}{817} \leqslant t_2 < \min \left(t_1, \frac{1}{2}(1 - t_1) \right) }} S\left(\mathcal{A}_{p_1 p_2}, p_2\right) \\
=&\ S_1 - S_2 + S_3.
\end{align}
Our aim is to show that the sparser set $\mathcal{A}$ contains the expected proportion of primes compared to the larger set $\mathcal{B}$, which requires us to decompose $S\left(\mathcal{A}, (2 X)^{\frac{1}{2}}\right)$ using the above Buchstab's identity, prove asymptotic formulas (for almost $x \in [X, 2X]$ except for $O\left(X (\log X)^{-B} \right)$ values) of the form
\begin{equation}
S\left(\mathcal{A}, z\right) = \eta (1+o(1)) S\left(\mathcal{B}, z\right)
\end{equation}
for some parts of it, and drop the remaining positive parts. 

In Sections 3 and 4 we provide some arithmetic information by proving mean value bounds for Dirichlet polynomials, and we shall use them to prove asymptotic formulas for terms of the form $S\left(\mathcal{A}_{p_{1} \ldots p_{n}}, X^{\delta}\right)$ and $S\left(\mathcal{A}_{p_{1} \ldots p_{n}}, p_{n} \right)$ in Section 5. In Section 6 we make further use of Buchstab's identity to decompose $S\left(\mathcal{A}, (2 X)^{\frac{1}{2}}\right)$ and prove Theorem~\ref{t1}.

\section{Arithmetic Information I}
In the following two sections we provide some arithmetic information which will help us prove the asymptotic formulas for sieve functions. In this section we only use the classical mean value estimate and Halász method. In next section we will provide a crucial result which will help us get asymptotic formulas for more sieve functions.

\begin{lemma}\label{l31}
Suppose that $M H = X$, $b_h = \Lambda(h)$ and $X^{\delta} \ll H \ll X^{\frac{79}{817}}$. Then for $(\log X)^{B/ \varepsilon} \leqslant T \leqslant X$, we have
$$
\left( \min \left(\eta, \frac{1}{T}\right) \right)^{2} \int_{T}^{2 T}|M(b+i t) H(b+i t)|^{2} d t \ll \eta^{2} (\log x)^{-10 B}.
$$
\end{lemma}
\begin{proof}
The proof is similar to that of [\cite{Jia120}, Lemma 1]. Let $s=b+i t$ and by the zero-free region of the $\zeta$ function, for $|t| \leqslant 2 X$ we have
$$
H(s)=\frac{\left(c_{2} H\right)^{1-s}-\left(c_{1} H\right)^{1-s}}{1-s}+O\left( (\log x)^{-2B/ \varepsilon} \right).
$$
So, for $T_{0} \leq|t| \leq 2 X$ we have $H(s) \ll (\log x)^{-B/ \varepsilon}$. According to the discussion in , there are $O\left( (\log X)^{2} \right)$ sets $S(V, W)$, where $S(V, W)$ is the set of $t_{k}(k=1, \ldots, K)$ with the property $\left|t_{r}-t_{s}\right| \geqslant 1$ $(r \neq s)$. Moreover,
$$
V \leqslant M^{\frac{1}{2}}\left|M\left(b+i t_{k}\right)\right|<2 V, \quad W \leqslant H^{\frac{1}{2}}\left|H\left(b+i t_{k}\right)\right|<2 W,
$$
where $X^{-1} \leqslant M^{-\frac{1}{2}} V, X^{-1} \leqslant H^{-\frac{1}{2}} W$ and $V \ll M^{\frac{1}{2}}, W \ll H^{\frac{1}{2}} (\log x)^{-B/ \varepsilon}$. Then we have
$$
\int_{T}^{2 T}|M(b+i t) H(b+i t)|^{2} d t \ll V^{2} W^{2} x^{-1} (\log x)^{2}|S(V, W)|+O\left(x^{-2 \varepsilon^2}\right),
$$
where $S(V, W)$ is one of sets with the above properties.

Assume $X^{\frac{1}{k+1}} \leqslant H<X^{\frac{1}{k}}$, where $k$ is a positive integer, $k \geqslant 10$ and $k \delta \ll 1$. Applying the mean value estimate to $M(s)$ and $H^{k}(s)$, we have
\begin{align}
\nonumber |S(V, W)| \ll&\ V^{-2}(M+T) (\log x)^{d}, \\
\nonumber |S(V, W)| \ll&\ W^{-2 k}\left(H^{k}+T\right) (\log x)^{d},
\end{align}
where $d=c / \delta^{2}$. Applying the Halász method to $M(s)$ and $H^{k}(s)$, we have
\begin{align}
\nonumber |S(V, W)| \ll&\ \left(V^{-2} M+V^{-6} M T\right) (\log x)^{d}, \\
\nonumber |S(V, W)| \ll&\ \left(W^{-2 k} H^{k}+W^{-6 k} H^{k} T\right) (\log x)^{d}.
\end{align}
Thus,
$$
V^{2} W^{2}|S(V, W)| \ll V^{2} W^{2} F (\log x)^{d},
$$
where
$$
F=\min \left\{V^{-2}(M+T), V^{-2} M+V^{-6} M T, W^{-2 k}(H^{k}+T), W^{-2 k} H^{k}+W^{-6 k} H^{k} T\right\} .
$$

It will be proved that
$$
\left( \min \left(\eta, \frac{1}{T}\right) \right)^{2} V^{2} W^{2} F \ll \eta^{2} x (\log x)^{-B/ \varepsilon}.
$$
We consider four cases.

(a) $F \leqslant 2 V^{-2} M, F \leqslant 2 W^{-2 k} H^{k}$. Then
\begin{align}
\nonumber V^{2} W^{2} F \ll&\ V^{2} W^{2} \min \left\{V^{-2} M, W^{-2 k} H^{k}\right\} \\
\nonumber \leqslant&\ V^{2} W^{2}\left(V^{-2} M\right)^{1-\frac{1}{2 k}}\left(W^{-2 k} H^{k}\right)^{\frac{1}{2 k}} \\
\nonumber =&\ V^{\frac{1}{k}} W M^{1-\frac{1}{2 k}} H^{\frac{1}{2}} \\
\nonumber \ll&\ x (\log x)^{-B/ \varepsilon}
\end{align}
and so
$$
\left( \min \left(\eta, \frac{1}{T}\right) \right)^{2} V^{2} W^{2} F \ll \eta^{2} x (\log x)^{-B/ \varepsilon}.
$$

(b) $F>2 V^{-2} M, F> 2 W^{-2 k} H^{k}$. Then
\begin{align}
\nonumber V^{2} W^{2} F \ll&\ V^{2} W^{2} \min \left\{V^{-2} T, V^{-6} M T, W^{-2 k} T, W^{-6 k} H^{k} T\right\} \\
\nonumber \leqslant&\ V^{2} W^{2}\left(V^{-2}\right)^{1-\frac{3}{2 k}}\left(V^{-6} M\right)^{\frac{1}{2 k}}\left(W^{-2 k}\right)^{\frac{1}{k}} T \\
\nonumber =&\ M^{\frac{1}{2 k}} T.
\end{align}
Since $k \geqslant 10$, we have $H \geqslant X^{\frac{1}{k+1}} \geqslant X^{1-\frac{k}{11}}, M^{\frac{1}{2 k}} \ll X^{\frac{1}{22}}$, and so
$$
\left( \min \left(\eta, \frac{1}{T}\right) \right)^{2} V^{2} W^{2} F \ll \frac{\eta}{T} x^{\frac{1}{22}} T \ll \eta^{2} x^{1-\varepsilon^2}.
$$

(c) $F \leqslant 2 V^{-2} M, F>2 W^{-2 k} H^{k}$. Then
\begin{align}
\nonumber V^{2} W^{2} F \ll&\ V^{2} W^{2} \min \left\{V^{-2} M, W^{-2 k} T, W^{-6 k} H^{k} T\right\} \\
\nonumber \leqslant&\ V^{2} W^{2}\left(V^{-2} M\right)^{1-\frac{1}{3 k}}\left(W^{-6 k} H^{k} T\right)^{\frac{1}{3 k}} \\
\nonumber \ll&\ M H^{\frac{1}{3}} T^{\frac{1}{3 k}},
\end{align}
since $V \ll M^{\frac{1}{2}}$. As $H \geqslant X^{\frac{1}{k+1}} \geqslant X^{\frac{21}{44k}+\varepsilon}$ and $M \leqslant X^{1-\frac{21}{44k}}$, we have
\begin{align}
\nonumber \left( \min \left(\eta, \frac{1}{T}\right) \right)^{2} V^{2} W^{2} F \ll&\ \eta^{2-\frac{1}{3k}} \left(\frac{1}{T}\right)^{\frac{1}{3k}} \left(M H\right)^{\frac{1}{3}} M^{\frac{2}{3}} T^{\frac{1}{3k}} \\
\nonumber \ll&\ x^{\frac{1}{3}-\frac{20.5}{21.5}(2-\frac{1}{3k})+\frac{2}{3}(1-\frac{21}{44k})} \\
\nonumber \ll&\ \eta^{2} x^{1-\varepsilon^2}.
\end{align}

(d) $F>2 V^{-2} M, F \leqslant 2 W^{-2 k} H^{k}$. Then
\begin{align}
\nonumber V^{2} W^{2} F \ll&\ V^{2} W^{2} \min \left\{V^{-2} T, V^{-6} M T, W^{-2 k} H^{k}\right\} \\
\nonumber \leqslant&\ V^{2} W^{2}\left(V^{-2} T\right)^{1-\frac{3}{2 k}}\left(V^{-6} M T\right)^{\frac{1}{2 k}}\left(W^{-2 k} H^{k}\right)^{\frac{1}{k}} \\
\nonumber =&\ M^{\frac{1}{2 k}} H T^{1-\frac{1}{k}}.
\end{align}
If $k \geqslant 11$, then $H \leqslant X^{\frac{1}{k}} \leqslant X^{1-\frac{21(k-1)}{11(2 k-1)}}, M \gg X^{\frac{21(k-1)}{11(2 k-1)}}$, and so
\begin{align}
\nonumber \left( \min \left(\eta, \frac{1}{T}\right) \right)^{2} V^{2} W^{2} F \ll&\ \eta^{1+\frac{1}{k}} \left(\frac{1}{T}\right)^{1-\frac{1}{k}} \left(M H\right) M^{-\frac{2k-1}{2k}} T^{1-\frac{1}{k}} \\
\nonumber \ll&\ x^{1-\frac{20.5}{21.5}(1+\frac{1}{k})-\frac{21(k-1)}{22k}} \\
\nonumber \ll&\ \eta^{2} x^{1-\varepsilon^2}.
\end{align}
If $k=10$, then $X^{\frac{1}{11}} \leqslant H \ll X^{\frac{79}{817}}, M \gg X^{\frac{738}{817}}$, and so
\begin{align}
\nonumber \left( \min \left(\eta, \frac{1}{T}\right) \right)^{2} V^{2} W^{2} F \ll&\ \eta^{\frac{11}{10}} \left(\frac{1}{T}\right)^{\frac{9}{10}} \left(M H\right) M^{-\frac{19}{20}} T^{\frac{9}{10}} \\
\nonumber \ll&\ x^{1-\frac{11}{10} \cdot \frac{20.5}{21.5}-\frac{19}{20} \cdot \frac{738}{817}} \\
\nonumber \ll&\ \eta^{2} x^{1-\varepsilon^2}.
\end{align}

Combining the above cases, Lemma~\ref{l31} is proved.
\end{proof}

\begin{lemma}\label{l32}
Suppose that $M H K=X$, and $M$ and $H$ satisfy one of the following 9 conditions:

(1) $M H \ll X^{\frac{673}{1247}},\ X^{\frac{82}{473}} \ll H,\ M^{29} H^{-1} \ll X^{\frac{427}{43}},\ X^{\frac{246}{817}} \ll M,\ M^{-1} H^{29} \ll X^{\frac{263}{43}},\ X^{\frac{246}{43}} \ll M^{12} H^{11}$;

(2) $M H \ll X^{\frac{223}{387}},\ M^{29} H^{19} \ll X^{\frac{632}{43}},\ X^{\frac{328}{387}} \ll M^{2} H,\ M^{2} H^{11} \ll X^{\frac{120}{43}},\ X^{\frac{1230}{43}} \ll M^{58} H^{49}$;

(3) $M H \ll X^{\frac{268}{473}},\ X^{\frac{41}{215}} \ll H,\ M^{6} H \ll X^{\frac{94}{43}},\ X^{\frac{82}{301}} \ll M \ll X^{\frac{16}{43}},\ M H^{8} \ll X^{\frac{98}{43}},\ X^{\frac{123}{43}} \ll M^{6} H^{5}$;

(4) $M H \ll X^{\frac{571}{817}},\ X^{\frac{82}{473}} \ll H,\ M^{12} H \ll X^{\frac{270}{43}},\ X^{\frac{574}{1247}} \ll M,\ M^{-1} H^{19} \ll X^{\frac{161}{43}},\ X^{\frac{328}{43}} \ll M^{12} H^{11}$;

(5) $M^{2} H \ll X^{\frac{446}{387}},\ M^{58} H^{9} \ll X^{\frac{1264}{43}},\ X^{\frac{164}{387}} \ll M,\ M H^{5} \ll X^{\frac{60}{43}},\ X^{\frac{615}{43}} \ll M^{29} H^{10}$;

(6) $X^{\frac{27}{43}} \ll M H \ll X^{\frac{219}{301}},\ X^{\frac{41}{215}} \ll H,\ M^{6} H \ll X^{\frac{135}{43}},\ X^{\frac{205}{473}} \ll M,\ M^{-1} H^{7} \ll X^{\frac{55}{43}},\ X^{\frac{164}{43}} \ll M^{6} H^{5}$;

(7) $M H \ll X^{\frac{268}{473}},\ M^{35} H^{23} \ll X^{\frac{767}{43}},\ X^{\frac{410}{473}} \ll M^{2} H,\ M^{2} H^{13} \ll X^{\frac{130}{43}},\ X^{\frac{1476}{43}} \ll M^{70} H^{59}$;

(8) $M H \ll X^{\frac{313}{559}},\ M^{41} H^{27} \ll X^{\frac{902}{43}},\ X^{\frac{492}{559}} \ll M^{2} H,\ M^{2} H^{15} \ll X^{\frac{138}{43}},\ X^{\frac{1722}{43}} \ll M^{82} H^{69}$;

(9) $M^{2} H \ll X^{\frac{536}{473}},\ M^{70} H^{11} \ll X^{\frac{1534}{43}},\ X^{\frac{205}{473}} \ll M,\ M H^{6} \ll X^{\frac{65}{43}},\ X^{\frac{738}{43}} \ll M^{35} H^{12}$.

Assume that for $(\log X)^{B/ \varepsilon} \leqslant |t| \leqslant 2 X, M(b+i t) \ll (\log x)^{-B/ \varepsilon}$ and $H(b+i t) \ll (\log x)^{-B/ \varepsilon}$. Then for $(\log X)^{B/ \varepsilon} \leqslant T \leqslant X$, we have
$$
\left( \min \left(\eta, \frac{1}{T}\right) \right)^{2} \int_{T}^{2 T}|M(b+i t) H(b+i t) K(b+i t)|^{2} d t \ll \eta^{2} (\log x)^{-10 B}.
$$
\end{lemma}
\begin{proof}
The proof is similar to that of [\cite{Jia120}, Lemmas 5,7,9,10,12]. Using the method of Lemma~\ref{l31}, we only need to show that for $T=1 / \eta=2 X^{\frac{20.5}{21.5}-\varepsilon}$,
$$
I=\int_{T}^{2 T}|M(b+i t) H(b+i t) K(b+i t)|^{2} d t \ll (\log x)^{-10 B}.
$$

Assuming the condition (1). By applying the mean value estimate and Halász method to $M^{3}(s), H^{5}(s)$ and $K^{2}(s)$, we get
$$
I \ll U^{2} V^{2} W^{2} x^{-1} F (\log x)^{c}
$$
where
\begin{align}
\nonumber F = \min & \left\{ V^{-6}\left(M^{3}+T\right), V^{-6} M^{3}+V^{-18} M^{3} T, W^{-10}\left(H^{5}+T\right), \right. \\
\nonumber & \left. W^{-10} H^{5}+W^{-30} H^{5} T, U^{-4}\left(K^{2}+T\right), U^{-4} K^{2}+U^{-12} K^{2} T \right\} .
\end{align}
We consider 8 cases:

(a) $F \leqslant 2 V^{-6} M^{3}, F \leqslant 2 W^{-10} H^{5}, F \leqslant 2 U^{-4} K^{2}$. Then
\begin{align}
\nonumber U^{2} V^{2} W^{2} F \ll&\ U^{2} V^{2} W^{2} \min \left\{V^{-4} W^{-2} M^{2} H, W^{-10} H^{5}, U^{-4} K^{2}\right\} \\
\nonumber \leqslant&\ U^{2} V^{2} W^{2}\left(V^{-4} W^{-2} M^{2} H\right)^{\frac{3}{8}}\left(W^{-10} H^{5}\right)^{\frac{1}{8}}\left(U^{-4} K^{2}\right)^{\frac{1}{2}} \\
\nonumber =&\ V^{\frac{1}{2}} M^{\frac{3}{4}} H K \\
\nonumber \ll&\ x (\log x)^{-11 B}.
\end{align}

(b) $F \leqslant 2 V^{-6} M^{3}, F \leqslant 2 W^{-10} H^{5}, F>2 U^{-4} K^{2}$. Then
\begin{align}
\nonumber U^{2} V^{2} W^{2} F \ll&\ U^{2} V^{2} W^{2} \min \left\{V^{-6} M^{3}, W^{-10} H^{5}, U^{-4} T, U^{-12} K^{2} T\right\} \\
\nonumber \leqslant&\ U^{2} V^{2} W^{2}\left(V^{-6} M^{3}\right)^{\frac{1}{3}}\left(W^{-10} H^{5}\right)^{\frac{1}{5}}\left(U^{-4} T\right)^{\frac{9}{20}}\left(U^{-12} K^{2} T\right)^{\frac{1}{60}} \\
\nonumber =&\ T^{\frac{7}{15}} M H K^{\frac{1}{30}} \\
\nonumber =&\ T^{\frac{7}{15}} (M H K)^{\frac{1}{30}} \left(M^{\frac{29}{30}} H^{\frac{29}{30}}\right)\\
\nonumber \ll&\ x^{1-\varepsilon^2},
\end{align}
where $M H \ll X^{\frac{673}{1247}}$ is required.

(c) $F \leqslant 2 V^{-6} M^{3}, F>2 W^{-10} H^{5}, F \leqslant 2 U^{-4} K^{2}$. Then
\begin{align}
\nonumber U^{2} V^{2} W^{2} F \ll&\ U^{2} V^{2} W^{2} \min \left\{V^{-6} M^{3}, W^{-10} T, W^{-30} H^{5} T, U^{-4} K^{2}\right\} \\
\nonumber \leqslant&\ U^{2} V^{2} W^{2}\left(V^{-6} M^{3}\right)^{\frac{1}{3}}\left(W^{-10} T\right)^{\frac{3}{20}}\left(W^{-30} H^{5} T\right)^{\frac{1}{60}}\left(U^{-4} K^{2}\right)^{\frac{1}{2}} \\
\nonumber =&\ T^{\frac{1}{6}} M H^{\frac{1}{12}} K \\
\nonumber =&\ T^{\frac{1}{6}} (M H K) H^{-\frac{11}{12}} \\
\nonumber \ll&\ x^{1-\varepsilon^2},
\end{align}
where $X^{\frac{82}{473}} \ll H$ is required.

(d) $F \leqslant 2 V^{-6} M^{3}, F>2 W^{-10} H^{5}, F> 2 U^{-4} K^{2}$. Then
\begin{align}
\nonumber U^{2} V^{2} W^{2} F \ll&\ U^{2} V^{2} W^{2} \min \left\{V^{-6} M^{3}, W^{-10} T, W^{-30} H^{5} T, U^{-4} T, U^{-12} K^{2} T\right\} \\
\nonumber \leqslant&\ U^{2} V^{2} W^{2}\left(V^{-6} M^{3}\right)^{\frac{1}{3}}\left(W^{-10} T\right)^{\frac{1}{5}}\left(U^{-4} T\right)^{\frac{9}{20}}\left(U^{-12} K^{2} T\right)^{\frac{1}{60}} \\
\nonumber =&\ T^{\frac{2}{3}} M K^{\frac{1}{30}} \\
\nonumber =&\ T^{\frac{2}{3}} (M H K)^{\frac{1}{30}} \left(M^{\frac{29}{30}} H^{-\frac{1}{30}}\right)\\
\nonumber \ll&\ x^{1-\varepsilon^2},
\end{align}
where $M^{29} H^{-1} \ll X^{\frac{427}{43}}$ is required.

(e) $F>2 V^{-6} M^{3}, F \leqslant 2 W^{-10} H^{5}, F \leqslant 2 U^{-4} K^{2}$. Then
\begin{align}
\nonumber U^{2} V^{2} W^{2} F \ll&\ U^{2} V^{2} W^{2} \min \left\{V^{-6} T, V^{-18} M^{3} T, W^{-10} H^{5}, U^{-4} K^{2}\right\} \\
\nonumber \leqslant&\ U^{2} V^{2} W^{2}\left(V^{-6} T\right)^{\frac{17}{60}}\left(V^{-18} M^{3} T\right)^{\frac{1}{60}}\left(W^{-10} H^{5}\right)^{\frac{1}{5}}\left(U^{-4} K^{2}\right)^{\frac{1}{2}} \\
\nonumber =&\ T^{\frac{3}{10}} M^{\frac{1}{20}} H K \\
\nonumber =&\ T^{\frac{3}{10}} (M H K) M^{-\frac{19}{20}} \\
\nonumber \ll&\ x^{1-\varepsilon^2},
\end{align}
where $X^{\frac{246}{817}} \ll M$ is required.

(f) $F>2 V^{-6} M^{3}, F \leqslant 2 W^{-10} H^{5}, F>2 U^{-4} K^{2}$. Then
\begin{align}
\nonumber U^{2} V^{2} W^{2} F \ll&\ U^{2} V^{2} W^{2} \min \left\{V^{-6} T, V^{-18} M^{3} T, W^{-10} H^{5}, U^{-4} T, U^{-12} K^{2} T\right\} \\
\nonumber \leqslant&\ U^{2} V^{2} W^{2}\left(V^{-6} T\right)^{\frac{1}{3}}\left(W^{-10} H^{5}\right)^{\frac{1}{5}}\left(U^{-4} T\right)^{\frac{9}{20}}\left(U^{-12} K^{2} T\right)^{\frac{1}{60}} \\
\nonumber =&\ T^{\frac{4}{5}} H K^{\frac{1}{30}} \\
\nonumber =&\ T^{\frac{4}{5}} (M H K)^{\frac{1}{30}} \left(M^{-\frac{1}{30}} H^{\frac{29}{30}} \right)\\
\nonumber \ll&\ x^{1-\varepsilon^2},
\end{align}
where $M^{-1} H^{29} \ll$ $X^{\frac{263}{43}}$ is required.

(g) $F>2 V^{-6} M^{3}, F> 2 W^{-10} H^{5}, F \leqslant 2 U^{-4} K^{2}$. Then
\begin{align}
\nonumber U^{2} V^{2} W^{2} F \ll&\ U^{2} V^{2} W^{2} \min \left\{V^{-6} T, V^{-18} M^{3} T, W^{-10} T, W^{-30} H^{5} T, U^{-4} K^{2}\right\} \\
\nonumber \leqslant&\ U^{2} V^{2} W^{2}\left(V^{-6} T\right)^{\frac{1}{3}}\left(W^{-10} T\right)^{\frac{3}{20}}\left(W^{-30} H^{5} T\right)^{\frac{1}{60}}\left(U^{-4} K^{2}\right)^{\frac{1}{2}} \\
\nonumber =&\ T^{\frac{1}{2}} H^{\frac{1}{12}} K \\
\nonumber =&\ T^{\frac{1}{2}} (M H K) \left( M^{-1} H^{-\frac{11}{12}} \right) \\
\nonumber \ll&\ x^{1-\varepsilon^2},
\end{align}
where $X^{\frac{246}{43}} \ll M^{12} H^{11}$ is required.

(h) $F>2 V^{-6} M^{3}, F> 2 W^{-10} H^{5}, F> 2 U^{-4} K^{2}$. Then
\begin{align}
\nonumber U^{2} V^{2} W^{2} F \ll&\ U^{2} V^{2} W^{2} \min \left\{V^{-6}, V^{-18} M^{3}, W^{-10}, W^{-30} H^{5}, U^{-4}, U^{-12} K^{2}\right\} T \\
\nonumber \leqslant&\ U^{2} V^{2} W^{2}\left(V^{-6}\right)^{\frac{170}{60}}\left(V^{-18} M^{3}\right)^{\frac{1}{60}}\left(W^{-10}\right)^{\frac{1}{5}}\left(U^{-4}\right)^{\frac{1}{2}} T \\
\nonumber =&\ T M^{\frac{1}{20}} \\
\nonumber \ll&\ x^{1-\varepsilon^2},
\end{align}
where $M \ll X^{\frac{5025}{13717}}$ and $T M^{\frac{1}{20}} \ll X^{\frac{20.5}{21.5}+\frac{1}{20} \cdot \frac{5025}{13717}} \ll X^{\frac{53321}{54868}}$ are used. The former one can be obtained by using $M H \ll X^{\frac{673}{1247}}$ and $X^{\frac{82}{473}} \ll H$.

Assuming the condition (2). By applying the mean value estimate and Halász method to $M^{2}(s) H(s), H^{5}(s)$ and $K^{2}(s)$, we get
$$
I \ll U^{2} V^{2} W^{2} x^{-1} F (\log x)^{c}
$$
where
\begin{align}
\nonumber F = \min & \left\{V^{-4} W^{-2}\left(M^{2} H+T\right), V^{-4} W^{-2} M^{2} H+V^{-12} W^{-6} M^{2} H T, \right. \\
\nonumber & \left. W^{-10}\left(H^{5}+T\right), W^{-10} H^{5}+W^{-30} H^{5} T, U^{-4}\left(K^{2}+T\right), U^{-4} K^{2}+U^{-12} K^{2} T\right\}.
\end{align}
We consider 8 cases:

(a) $F \leqslant 2 V^{-4} W^{-2} M^{2} H, F \leqslant 2 W^{-10} H^{5}, F \leqslant 2 U^{-4} K^{2}$. Then
\begin{align}
\nonumber U^{2} V^{2} W^{2} F \ll&\ U^{2} V^{2} W^{2} \min \left\{V^{-4} W^{-2} M^{2} H, W^{-10} H^{5}, U^{-4} K^{2}\right\} \\
\nonumber \leqslant&\ U^{2} V^{2} W^{2}\left(V^{-4} W^{-2} M^{2} H\right)^{\frac{3}{8}}\left(W^{-10} H^{5}\right)^{\frac{1}{8}}\left(U^{-4} K^{2}\right)^{\frac{1}{2}} \\
\nonumber =&\ V^{\frac{1}{2}} M^{\frac{3}{4}} H K \\
\nonumber \ll&\ x (\log x)^{-11 B}.
\end{align}

(b) $F \leqslant 2 V^{-4} W^{-2} M^{2} H, F \leqslant 2 W^{-10} H^{5}, F>2 U^{-4} K^{2}$. Then
\begin{align}
\nonumber U^{2} V^{2} W^{2} F \ll&\ U^{2} V^{2} W^{2} \min \left\{V^{-4} W^{-2} M^{2} H, W^{-10} H^{5}, U^{-4} T, U^{-12} K^{2} T\right\} \\
\nonumber \leqslant&\ U^{2} V^{2} W^{2}\left(V^{-4} W^{-2} M^{2} H\right)^{\frac{1}{2}}\left(W^{-10} H^{5}\right)^{\frac{1}{10}}\left(U^{-4} T\right)^{\frac{7}{20}}\left(U^{-12} K^{2} T\right)^{\frac{1}{20}} \\
\nonumber =&\ T^{\frac{2}{5}} M H K^{\frac{1}{10}} \\
\nonumber =&\ T^{\frac{2}{5}} (M H K)^{\frac{1}{10}} \left(M^{\frac{9}{10}} H^{\frac{9}{10}} \right) \\
\nonumber \ll&\ x^{1-\varepsilon^2},
\end{align}
where $M H \ll X^{\frac{223}{387}}$ is required.

(c) $F \leqslant 2 V^{-4} W^{-2} M^{2} H, F>2 W^{-10} H^{5}, F \leqslant 2 U^{-4} K^{2}$. Then
\begin{align}
\nonumber U^{2} V^{2} W^{2} F \ll&\ U^{2} V^{2} W^{2} \min \left\{V^{-4} W^{-2} M^{2} H, W^{-10} T, W^{-30} H^{5} T, U^{-4} K^{2}\right\} \\
\nonumber \leqslant&\ U^{2} V^{2} W^{2}\left(V^{-4} W^{-2} M^{2} H\right)^{\frac{1}{2}}\left(U^{-4} K^{2}\right)^{\frac{1}{2}} \\
\nonumber =&\ W M H^{\frac{1}{2}} K \\
\nonumber \ll&\ x (\log x)^{-11 B}.
\end{align}

(d) $F \leqslant 2 V^{-4} W^{-2} M^{2} H, F>2 W^{-10} H^{5}, F> 2 U^{-4} K^{2}$. Then
\begin{align}
\nonumber U^{2} V^{2} W^{2} F \ll&\ U^{2} V^{2} W^{2} \min \left\{V^{-4} W^{-2} M^{2} H, W^{-10} T, W^{-30} H^{5} T, U^{-4} T, U^{-12} K^{2} T\right\} \\
\nonumber \leqslant&\ U^{2} V^{2} W^{2}\left(V^{-4} W^{-2} M^{2} H\right)^{\frac{1}{2}}\left(W^{-30} H^{5} T\right)^{\frac{1}{30}}\left(U^{-4} T\right)^{\frac{9}{20}}\left(U^{-12} K^{2} T\right)^{\frac{1}{10}} \\
\nonumber =&\ T^{\frac{1}{2}} M H^{\frac{2}{3}} K^{\frac{1}{30}} \\
\nonumber =&\ T^{\frac{1}{2}} (M H K)^{\frac{1}{30}} \left(M^{\frac{29}{30}} H^{\frac{19}{30}}\right) \\
\nonumber \ll&\ x^{1-\varepsilon^2},
\end{align}
where $M^{29} H^{19} \ll X^{\frac{632}{43}}$ is required.

(e) $F>2 V^{-4} W^{-2} M^{2} H, F \leqslant 2 W^{-10} H^{5}, F \leqslant 2 U^{-4} K^{2}$. Then
\begin{align}
\nonumber U^{2} V^{2} W^{2} F \ll&\ U^{2} V^{2} W^{2} \min \left\{V^{-4} W^{-2} T, V^{-12} W^{-6} M^{2} H T, W^{-10} H^{5}, U^{-4} K^{2}\right\} \\
\nonumber \leqslant&\ U^{2} V^{2} W^{2}\left(V^{-4} W^{-2} T\right)^{\frac{7}{20}}\left(V^{-12} W^{-6} M^{2} H T\right)^{\frac{1}{20}}\left(W^{-10} H^{5}\right)^{\frac{1}{10}}\left(U^{-4} K^{2}\right)^{\frac{1}{2}} \\
\nonumber =&\ T^{\frac{2}{5}} M^{\frac{1}{10}} H^{\frac{11}{20}} K \\
\nonumber =&\ T^{\frac{2}{5}} (M H K) \left( M^{-\frac{9}{10}} H^{-\frac{9}{20}} \right)\\
\nonumber \ll&\ x^{1-\varepsilon^2},
\end{align}
where $X^{\frac{328}{387}} \ll M^{2} H$ is required.

(f) $F>2 V^{-4} W^{-2} M^{2} H, F \leqslant 2 W^{-10} H^{5}, F>2 U^{-4} K^{2}$. Then
\begin{align}
\nonumber U^{2} V^{2} W^{2} F \ll&\ U^{2} V^{2} W^{2} \min \left\{V^{-4} W^{-2} T, V^{-12} W^{-6} M^{2} H T, W^{-10} H^{5}, U^{-4} T, U^{-12} K^{2} T\right\} \\
\nonumber \leqslant&\ U^{2} V^{2} W^{2}\left(V^{-4} W^{-2} T\right)^{\frac{7}{20}}\left(V^{-12} W^{-6} M^{2} H T\right)^{\frac{1}{20}}\left(W^{-10} H^{5}\right)^{\frac{1}{10}}\left(U^{-4} T\right)^{\frac{1}{2}} \\
\nonumber =&\ T^{\frac{9}{10}} \left( M^{\frac{1}{10}} H^{\frac{11}{20}} \right)\\
\nonumber \ll&\ x^{1-\varepsilon^2},
\end{align}
where $M^{2} H^{11} \ll X^{\frac{122}{43}}$ is required. This can be obtained by using $M^{2} H^{11} \ll X^{\frac{120}{43}}$.

(g) $F>2 V^{-4} W^{-2} M^{2} H, F> 2 W^{-10} H^{5}, F \leqslant 2 U^{-4} K^{2}$. Then
\begin{align}
\nonumber U^{2} V^{2} W^{2} F \ll&\ U^{2} V^{2} W^{2} \min \left\{V^{-4} W^{-2} T, V^{-12} W^{-6} M^{2} H T, W^{-10} T, W^{-30} H^{5} T, U^{-4} K^{2}\right\} \\
\nonumber \leqslant&\ U^{2} V^{2} W^{2}\left(V^{-4} W^{-2} T\right)^{\frac{9}{20}}\left(V^{-12} W^{-6} M^{2} H T\right)^{\frac{1}{60}}\left(W^{-30} H^{5} T\right)^{\frac{1}{30}}\left(U^{-4} K^{2}\right)^{\frac{1}{2}} \\
\nonumber =&\ T^{\frac{1}{2}} M^{\frac{1}{30}} H^{\frac{11}{60}} K \\
\nonumber =&\ T^{\frac{1}{2}} (M H K) \left( M^{-\frac{29}{30}} H^{-\frac{49}{60}} \right) \\
\nonumber \ll&\ x^{1-\varepsilon^2},
\end{align}
where $X^{\frac{1230}{43}} \ll M^{58} H^{49}$ is required.

(h) $F>2 V^{-4} W^{-2} M^{2} H, F> 2 W^{-10} H^{5}, F> 2 U^{-4} K^{2}$. Then
\begin{align}
\nonumber U^{2} V^{2} W^{2} F \ll&\ U^{2} V^{2} W^{2} \min \left\{V^{-4} W^{-2}, V^{-12} W^{-6} M^{2} H, W^{-10}, W^{-30} H^{5}, U^{-4}, U^{-12} K^{2} \right\} T \\
\nonumber \leqslant&\ U^{2} V^{2} W^{2}\left(V^{-4} W^{-2}\right)^{\frac{9}{20}}\left(V^{-12} W^{-6} M^{2} H\right)^{\frac{1}{60}}\left(W^{-30} H^{5}\right)^{\frac{1}{30}}\left(U^{-4}\right)^{\frac{1}{2}} T \\
\nonumber =&\ T \left( M^{\frac{1}{30}} H^{\frac{11}{60}} \right) \\
\nonumber \ll&\ x^{1-\varepsilon^2},
\end{align}
where $M^{2} H^{11} \ll X^{\frac{120}{43}}$ is required.

Assuming the condition (3). By applying the mean value estimate and Halász method to $M^{3}(s), H^{4}(s)$ and $K^{2}(s)$, we get
$$
I \ll U^{2} V^{2} W^{2} x^{-1} F (\log x)^{c}
$$
where
\begin{align}
\nonumber F = \min & \left\{ V^{-6}\left(M^{3}+T\right), V^{-6} M^{3}+V^{-18} M^{3} T, W^{-8}\left(H^{4}+T\right), \right. \\
\nonumber & \left. W^{-8} H^{4} + W^{-24} H^{4} T, U^{-4}\left(K^{2}+T\right), U^{-4} K^{2}+U^{-12} K^{2} T\right\}.
\end{align}
We consider 8 cases:

(a) $F \leqslant 2 V^{-6} M^{3}, F \leqslant 2 W^{-8} H^{4}, F \leqslant 2 U^{-4} K^{2}$. Then
\begin{align}
\nonumber U^{2} V^{2} W^{2} F \ll&\ U^{2} V^{2} W^{2} \min \left\{V^{-6} M^{3}, W^{-8} H^{4}, U^{-4} K^{2}\right\} \\
\nonumber \leqslant&\ U^{2} V^{2} W^{2}\left(V^{-6} M^{3}\right)^{\frac{1}{4}}\left(W^{-8} H^{4}\right)^{\frac{1}{4}}\left(U^{-4} K^{2}\right)^{\frac{1}{2}} \\
\nonumber =&\ V^{\frac{1}{2}} M^{\frac{3}{4}} H K \\
\nonumber \ll&\ x (\log x)^{-11 B}.
\end{align}

(b) $F \leqslant 2 V^{-6} M^{3}, F \leqslant 2 W^{-8} H^{4}, F>2 U^{-4} K^{2}$. Then
\begin{align}
\nonumber U^{2} V^{2} W^{2} F \ll&\ U^{2} V^{2} W^{2} \min \left\{V^{-6} M^{3}, W^{-8} H^{4}, U^{-4} T, U^{-12} K^{2} T\right\} \\
\nonumber \leqslant&\ U^{2} V^{2} W^{2}\left(V^{-6} M^{3}\right)^{\frac{1}{3}}\left(W^{-8} H^{4}\right)^{\frac{1}{4}}\left(U^{-4} T\right)^{\frac{3}{8}}\left(U^{-12} K^{2} T\right)^{\frac{1}{24}} \\
\nonumber =&\ T^{\frac{5}{12}} M H K^{\frac{1}{12}} \\
\nonumber =&\ T^{\frac{5}{12}} (M H K)^{\frac{1}{12}} \left(M^{\frac{11}{12}} H^{\frac{11}{12}} \right) \\
\nonumber \ll&\ x^{1-\varepsilon^2},
\end{align}
where $M H \ll X^{\frac{268}{473}}$ is required.

(c) $F \leqslant 2 V^{-6} M^{3}, F>2 W^{-8} H^{4}, F \leqslant 2 U^{-4} K^{2}$. Then
\begin{align}
\nonumber U^{2} V^{2} W^{2} F \ll&\ U^{2} V^{2} W^{2} \min \left\{V^{-6} M^{3}, W^{-8} T, W^{-24} H^{4} T, U^{-4} K^{2}\right\} \\
\nonumber \leqslant&\ U^{2} V^{2} W^{2}\left(V^{-6} M^{3}\right)^{\frac{1}{3}}\left(W^{-8} T\right)^{\frac{1}{8}}\left(W^{-24} H^{4} T\right)^{\frac{1}{24}}\left(U^{-4} K^{2}\right)^{\frac{1}{2}} \\
\nonumber =&\ T^{\frac{1}{6}} M H^{\frac{1}{6}} K \\
\nonumber =&\ T^{\frac{1}{6}} (M H K) H^{-\frac{5}{6}} \\
\nonumber \ll&\ x^{1-\varepsilon^2},
\end{align}
where $X^{\frac{41}{215}} \ll H$ is required.

(d) $F \leqslant 2 V^{-6} M^{3}, F>2 W^{-8} H^{4}, F>2 U^{-4} K^{2}$. Then
\begin{align}
\nonumber U^{2} V^{2} W^{2} F \ll&\ U^{2} V^{2} W^{2} \min \left\{V^{-6} M^{3}, W^{-8} T, W^{-24} H^{4} T, U^{-4} T, U^{-12} K^{2} T\right\} \\
\nonumber \leqslant&\ U^{2} V^{2} W^{2}\left(V^{-6} M^{3}\right)^{\frac{1}{3}}\left(W^{-8} T\right)^{\frac{1}{8}}\left(W^{-24} H^{4} T\right)^{\frac{1}{24}}\left(U^{-4} T\right)^{\frac{1}{2}} \\
\nonumber =&\ T^{\frac{2}{3}} \left( M H^{\frac{1}{6}} \right) \\
\nonumber \ll&\ x^{1-\varepsilon^2},
\end{align}
where $M^{6} H \ll X^{\frac{94}{43}}$ is required.

(e) $F>2 V^{-6} M^{3}, F \leqslant 2 W^{-8} H^{4}, F \leqslant 2 U^{-4} K^{2}$. Then
\begin{align}
\nonumber U^{2} V^{2} W^{2} F \ll&\ U^{2} V^{2} W^{2} \min \left\{V^{-6} T, V^{-18} M^{3} T, W^{-8} H^{4}, U^{-4} K^{2}\right\} \\
\nonumber \leqslant&\ U^{2} V^{2} W^{2}\left(V^{-6} T\right)^{\frac{5}{24}}\left(V^{-18} M^{3} T\right)^{\frac{1}{24}}\left(W^{-8} H^{4}\right)^{\frac{1}{4}}\left(U^{-4} K^{2}\right)^{\frac{1}{2}} \\
\nonumber =&\ T^{\frac{1}{4}} M^{\frac{1}{8}} H K \\
\nonumber =&\ T^{\frac{1}{4}} (M H K) M^{-\frac{7}{8}} \\
\nonumber \ll&\ x^{1-\varepsilon^2},
\end{align}
where $X^{\frac{82}{301}} \ll M$ is required.

(f) $F>2 V^{-6} M^{3}, F \leqslant 2 W^{-8} H^{4}, F>2 U^{-4} K^{2}$. Then
\begin{align}
\nonumber U^{2} V^{2} W^{2} F \ll&\ U^{2} V^{2} W^{2} \min \left\{V^{-6} T, V^{-18} M^{3} T, W^{-8} H^{4}, U^{-4} T, U^{-12} K^{2} T\right\} \\
\nonumber \leqslant&\ U^{2} V^{2} W^{2}\left(V^{-6} T\right)^{\frac{5}{24}}\left(V^{-18} M^{3} T\right)^{\frac{1}{24}}\left(W^{-8} H^{4}\right)^{\frac{1}{4}}\left(U^{-4} T\right)^{\frac{1}{2}} \\
\nonumber =&\ T^{\frac{3}{4}} \left( M^{\frac{1}{8}} H \right) \\
\nonumber \ll&\ x^{1-\varepsilon^2},
\end{align}
where $M H^{8} \ll X^{\frac{98}{43}}$ is required.

(g) $F>2 V^{-6} M^{3}, F> 2 W^{-8} H^{4}, F \leqslant 2 U^{-4} K^{2}$. Then
\begin{align}
\nonumber U^{2} V^{2} W^{2} F \ll&\ U^{2} V^{2} W^{2} \min \left\{V^{-6} T, V^{-18} M^{3} T, W^{-8} T, W^{-24} H^{4} T, U^{-4} K^{2}\right\} \\
\nonumber \leqslant&\ U^{2} V^{2} W^{2}\left(V^{-6} T\right)^{\frac{1}{3}}\left(W^{-8} T\right)^{\frac{1}{8}}\left(W^{-24} H^{4} T\right)^{\frac{1}{24}}\left(U^{-4} K^{2}\right)^{\frac{1}{2}} \\
\nonumber =&\ T^{\frac{1}{2}} H^{\frac{1}{6}} K \\
\nonumber =&\ T^{\frac{1}{2}} (M H K) \left( M^{-1} H^{-\frac{5}{6}} \right) \\
\nonumber \ll&\ x^{1-\varepsilon^2},
\end{align}
where $X^{\frac{123}{43}} \ll M^{6} H^{5}$ is required.

(h) $F>2 V^{-6} M^{3}, F> 2 W^{-8} H^{4}, F> 2 U^{-4} K^{2}$. Then
\begin{align}
\nonumber U^{2} V^{2} W^{2} F \ll&\ U^{2} V^{2} W^{2} \min \left\{V^{-6}, V^{-18} M^{3}, W^{-8}, W^{-24} H^{4}, U^{-4}, U^{-12} K^{2}\right\} T \\
\nonumber \leqslant&\ U^{2} V^{2} W^{2}\left(V^{-6}\right)^{\frac{5}{24}}\left(V^{-18} M^{3}\right)^{\frac{1}{24}}\left(W^{-8}\right)^{\frac{1}{4}}\left(U^{-4}\right)^{\frac{1}{2}} T \\
\nonumber =&\ T M^{\frac{1}{8}} \\
\nonumber \ll&\ x^{1-\varepsilon^2},
\end{align}
where $M \ll X^{\frac{16}{43}}$ is required.

Assuming the condition (4). By applying the mean value estimate and Halász method to $M^{2}(s), H^{5}(s)$ and $K^{3}(s)$, we get
$$
I \ll U^{2} V^{2} W^{2} x^{-1} F (\log x)^{c}
$$
where
\begin{align}
\nonumber F = \min & \left\{ V^{-4}\left(M^{2}+T\right), V^{-4} M^{2}+V^{-12} M^{2} T, W^{-10}\left(H^{5}+T\right), \right. \\
\nonumber & \left. W^{-10} H^{5} + W^{-30} H^{5} T, U^{-6}\left(K^{3}+T\right), U^{-6} K^{3}+U^{-18} K^{3} T\right\}.
\end{align}
We consider 8 cases:

(a) $F \leqslant 2 V^{-4} M^{2}, F \leqslant 2 W^{-10} H^{5}, F \leqslant 2 U^{-6} K^{3}$. Then
\begin{align}
\nonumber U^{2} V^{2} W^{2} F \ll&\ U^{2} V^{2} W^{2} \min \left\{V^{-4} M^{2}, W^{-10} H^{5}, U^{-6} K^{3}\right\} \\
\nonumber \leqslant&\ U^{2} V^{2} W^{2}\left(V^{-4} M^{2}\right)^{\frac{1}{2}}\left(W^{-10} H^{5}\right)^{\frac{1}{6}}\left(U^{-6} K^{3}\right)^{\frac{1}{3}} \\
\nonumber =&\ W^{\frac{1}{3}} M H^{\frac{5}{6}} K \\
\nonumber \ll&\ x (\log x)^{-11 B}.
\end{align}

(b) $F \leqslant 2 V^{-4} M^{2}, F \leqslant 2 W^{-10} H^{5}, F>2 U^{-6} K^{3}$. Then
\begin{align}
\nonumber U^{2} V^{2} W^{2} F \ll&\ U^{2} V^{2} W^{2} \min \left\{V^{-4} M^{2}, W^{-10} H^{5}, U^{-6} T, U^{-18} K^{3} T\right\} \\
\nonumber \leqslant&\ U^{2} V^{2} W^{2}\left(V^{-4} M^{2}\right)^{\frac{1}{2}}\left(W^{-10} H^{5}\right)^{\frac{1}{5}}\left(U^{-6} T\right)^{\frac{17}{60}}\left(U^{-18} K^{3} T\right)^{\frac{1}{60}} \\
\nonumber =&\ T^{\frac{3}{10}} M H K^{\frac{1}{20}} \\
\nonumber =&\ T^{\frac{3}{10}} (M H K)^{\frac{1}{20}} \left(M^{\frac{19}{20}} H^{\frac{19}{20}} \right) \\
\nonumber \ll&\ x^{1-\varepsilon^2},
\end{align}
where $M H \ll X^{\frac{571}{817}}$ is required.

(c) $F \leqslant 2 V^{-4} M^{2}, F>2 W^{-10} H^{5}, F \leqslant 2 U^{-6} K^{3}$. Then
\begin{align}
\nonumber U^{2} V^{2} W^{2} F \ll&\ U^{2} V^{2} W^{2} \min \left\{V^{-4} M^{2}, W^{-10} T, W^{-30} H^{5} T, U^{-6} K^{3}\right\} \\
\nonumber \leqslant&\ U^{2} V^{2} W^{2}\left(V^{-4} M^{2}\right)^{\frac{1}{2}}\left(W^{-10} T\right)^{\frac{3}{20}}\left(W^{-30} H^{5} T\right)^{\frac{1}{60}}\left(U^{-6} K^{3}\right)^{\frac{1}{3}} \\
\nonumber =&\ T^{\frac{1}{6}} M H^{\frac{1}{12}} K \\
\nonumber =&\ T^{\frac{1}{6}} (M H K) H^{-\frac{11}{12}} \\
\nonumber \ll&\ x^{1-\varepsilon^2},
\end{align}
where $X^{\frac{82}{473}} \ll H$ is required.

(d) $F \leqslant 2 V^{-4} M^{2}, F>2 W^{-10} H^{5}, F> 2 U^{-6} K^{3}$. Then
\begin{align}
\nonumber U^{2} V^{2} W^{2} F \ll&\ U^{2} V^{2} W^{2} \min \left\{V^{-4} M^{2}, W^{-10} T, W^{-30} H^{5} T, U^{-6} T, U^{-18} K^{3} T\right\} \\
\nonumber \leqslant&\ U^{2} V^{2} W^{2}\left(V^{-4} M^{2}\right)^{\frac{1}{2}}\left(W^{-10} T\right)^{\frac{3}{20}}\left(W^{-30} H^{5} T\right)^{\frac{1}{60}}\left(U^{-6} T\right)^{\frac{1}{3}} \\
\nonumber =&\ T^{\frac{1}{2}} \left( M H^{\frac{1}{12}} \right) \\
\nonumber \ll&\ x^{1-\varepsilon^2},
\end{align}
where $M^{12} H \ll X^{\frac{270}{43}}$ is required.

(e) $F>2 V^{-4} M^{2}, F \leqslant 2 W^{-10} H^{5}, F \leqslant 2 U^{-6} K^{3}$. Then
\begin{align}
\nonumber U^{2} V^{2} W^{2} F \ll&\ U^{2} V^{2} W^{2} \min \left\{V^{-4} T, V^{-12} M^{2} T, W^{-10} H^{5}, U^{-6} K^{3}\right\} \\
\nonumber \leqslant&\ U^{2} V^{2} W^{2}\left(V^{-4} T\right)^{\frac{9}{20}}\left(V^{-12} M^{2} T\right)^{\frac{1}{60}}\left(W^{-10} H^{5}\right)^{\frac{1}{5}}\left(U^{-6} K^{3}\right)^{\frac{1}{3}} \\
\nonumber =&\ T^{\frac{7}{15}} M^{\frac{1}{30}} H K \\
\nonumber =&\ T^{\frac{7}{15}} (M H K) M^{-\frac{29}{30}} \\
\nonumber \ll&\ x^{1-\varepsilon^2},
\end{align}
where $X^{\frac{574}{1247}} \ll M$ is required.

(f) $F>2 V^{-4} M^{2}, F \leqslant 2 W^{-10} H^{5}, F>2 U^{-6} K^{3}$. Then
\begin{align}
\nonumber U^{2} V^{2} W^{2} F \ll&\ U^{2} V^{2} W^{2} \min \left\{V^{-4} T, V^{-12} M^{2} T, W^{-10} H^{5}, U^{-6} T, U^{-18} K^{3} T\right\} \\
\nonumber \leqslant&\ U^{2} V^{2} W^{2}\left(V^{-4} T\right)^{\frac{1}{2}}\left(W^{-10} H^{5}\right)^{\frac{1}{5}}\left(U^{-6} T\right)^{\frac{17}{60}}\left(U^{-18} K^{3} T\right)^{\frac{1}{60}} \\
\nonumber =&\ T^{\frac{4}{5}} H K^{\frac{1}{20}} \\
\nonumber =&\ T^{\frac{4}{5}} (M H K)^{\frac{1}{20}} \left(M^{-\frac{1}{20}} H^{\frac{19}{20}} \right) \\
\nonumber \ll&\ x^{1-\varepsilon^2},
\end{align}
where $M^{-1} H^{19} \ll X^{\frac{161}{43}}$ is required.

(g) $F>2 V^{-4} M^{2}, F> 2 W^{-10} H^{5}, F \leqslant 2 U^{-6} K^{3}$. Then
\begin{align}
\nonumber U^{2} V^{2} W^{2} F \ll&\ U^{2} V^{2} W^{2} \min \left\{V^{-4} T, V^{-12} M^{2} T, W^{-10} T, W^{-30} H^{5} T, U^{-6} K^{3}\right\} \\
\nonumber \leqslant&\ U^{2} V^{2} W^{2}\left(V^{-4} T\right)^{\frac{1}{2}}\left(W^{-10} T\right)^{\frac{3}{20}}\left(W^{-30} H^{5} T\right)^{\frac{1}{60}}\left(U^{-6} K^{3}\right)^{\frac{1}{3}} \\
\nonumber =&\ T^{\frac{2}{3}} H^{\frac{1}{12}} K \\
\nonumber =&\ T^{\frac{2}{3}} (M H K) \left( M^{-1} H^{-\frac{11}{12}} \right) \\
\nonumber \ll&\ x^{1-\varepsilon^2},
\end{align}
where $X^{\frac{328}{43}} \ll M^{12} H^{11}$ is required.

(h) $F>2 V^{-4} M^{2}, F> 2 W^{-10} H^{5}, F> 2 U^{-6} K^{3}$. Then
\begin{align}
\nonumber U^{2} V^{2} W^{2} F \ll&\ U^{2} V^{2} W^{2} \min \left\{V^{-4}, V^{-12} M^{2}, W^{-10}, W^{-30} H^{5}, U^{-6}, U^{-18} K^{3}\right\} T \\
\nonumber \leqslant&\ U^{2} V^{2} W^{2}\left(V^{-4}\right)^{\frac{1}{2}}\left(W^{-10}\right)^{\frac{1}{5}}\left(U^{-6}\right)^{\frac{17}{60}}\left(U^{-18} K^{3}\right)^{\frac{1}{60}} T \\
\nonumber =&\ T K^{\frac{1}{20}} \\
\nonumber =&\ T (M H K)^{\frac{1}{20}} \left(M^{-\frac{1}{20}} H^{-\frac{1}{20}} \right)\\
\nonumber \ll&\ x^{1-\varepsilon^2},
\end{align}
where $X^{\frac{3}{43}} \ll M H$ is required. This can be obtained by using $X^{\frac{574}{1247}} \ll M$ and $X^{\frac{82}{473}} \ll H$.

Assuming the condition (5). By applying the mean value estimate and Halász method to $M^{2}(s), H^{5}(s)$ and $K^{2}(s) H(s)$, we get
$$
I \ll U^{2} V^{2} W^{2} x^{-1} F (\log x)^{c}
$$
where
\begin{align}
\nonumber F = \min & \left\{ V^{-4}\left(M^{2}+T\right), V^{-4} M^{2}+V^{-12} M^{2} T, W^{-10}\left(H^{5}+T\right), \right. \\
\nonumber & \left. W^{-10} H^{5}+W^{-30} H^{5} T, U^{-4} W^{-2}\left(K^{2} H+T\right), U^{-4} W^{-2} K^{2} H +U^{-12} W^{-6} K^{2} H T\right\}.
\end{align}
We consider 8 cases:

(a) $F \leqslant 2 V^{-4} M^{2}, F \leqslant 2 W^{-10} H^{5}, F \leqslant 2 U^{-4} W^{-2} K^{2} H$. Then
\begin{align}
\nonumber U^{2} V^{2} W^{2} F \ll&\ U^{2} V^{2} W^{2} \min \left\{V^{-4} M^{2}, W^{-10} H^{5}, U^{-4} W^{-2} K^{2} H\right\} \\
\nonumber \leqslant&\ U^{2} V^{2} W^{2}\left(V^{-4} M^{2}\right)^{\frac{1}{2}}\left(U^{-4} W^{-2} K^{2} H\right)^{\frac{1}{2}} \\
\nonumber =&\ W M H^{\frac{1}{2}} K \\
\nonumber \ll&\ x (\log x)^{-11 B}.
\end{align}

(b) $F \leqslant 2 V^{-4} M^{2}, F \leqslant 2 W^{-10} H^{5}, F>2 U^{-4} W^{-2} K^{2} H$. Then
\begin{align}
\nonumber U^{2} V^{2} W^{2} F \ll&\ U^{2} V^{2} W^{2} \min \left\{V^{-4} M^{2}, W^{-10} H^{5}, U^{-4} W^{-2} T, U^{-12} W^{-6} K^{2} H T\right\} \\
\nonumber \leqslant&\ U^{2} V^{2} W^{2}\left(V^{-4} M^{2}\right)^{\frac{1}{2}}\left(W^{-10} H^{5}\right)^{\frac{1}{10}}\left(U^{-4} W^{-2} T\right)^{\frac{7}{20}}\left(U^{-12} W^{-6} K^{2} H T\right)^{\frac{1}{20}}\\
\nonumber =&\ T^{\frac{2}{5}} M H^{\frac{11}{20}} K^{\frac{1}{10}} \\
\nonumber =&\ T^{\frac{2}{5}} (M H K)^{\frac{1}{10}} \left(M^{\frac{9}{10}} H^{\frac{9}{20}} \right) \\
\nonumber \ll&\ x^{1-\varepsilon^2},
\end{align}
where $M^{2} H \ll X^{\frac{446}{387}}$ is required.

(c) $F \leqslant 2 V^{-4} M^{2}, F>2 W^{-10} H^{5}, F \leqslant 2 U^{-4} W^{-2} K^{2} H$. Then
\begin{align}
\nonumber U^{2} V^{2} W^{2} F \ll&\ U^{2} V^{2} W^{2} \min \left\{V^{-4} M^{2}, W^{-10} T, W^{-30} H^{5} T, U^{-4} W^{-2} K^{2} H\right\} \\
\nonumber \leqslant&\ U^{2} V^{2} W^{2}\left(V^{-4} M^{2}\right)^{\frac{1}{2}}\left(U^{-4} W^{-2} K^{2} H\right)^{\frac{1}{2}} \\
\nonumber =&\ W M H^{\frac{1}{2}} K \\
\nonumber \ll&\ x (\log x)^{-11 B}.
\end{align}

(d) $F \leqslant 2 V^{-4} M^{2}, F>2 W^{-10} H^{5}, F> 2 U^{-4} W^{-2} K^{2} H$. Then
\begin{align}
\nonumber U^{2} V^{2} W^{2} F \ll&\  U^{2} V^{2} W^{2} \min \left\{V^{-4} M^{2}, W^{-10} T, W^{-30} H^{5} T, U^{-4} W^{-2} T, U^{-12} W^{-6} K^{2} H T\right\} \\
\nonumber \leqslant&\ U^{2} V^{2} W^{2}\left(V^{-4} M^{2}\right)^{\frac{1}{2}}\left(W^{-30} H^{5} T\right)^{\frac{1}{30}}\left(U^{-4} W^{-2} T\right)^{\frac{9}{20}}\left(U^{-12} W^{-6} K^{2} H T\right)^{\frac{1}{60}}\\
\nonumber =&\ T^{\frac{1}{2}} M H^{\frac{11}{60}} K^{\frac{1}{30}} \\
\nonumber =&\ T^{\frac{1}{2}} (M H K)^{\frac{1}{30}} \left(M^{\frac{29}{30}} H^{\frac{9}{60}}\right) \\
\nonumber \ll&\ x^{1-\varepsilon^2},
\end{align}
where $M^{58} H^{9} \ll X^{\frac{1264}{43}}$ is required.

(e) $F>2 V^{-4} M^{2}, F \leqslant 2 W^{-10} H^{5}, F \leqslant 2 U^{-4} W^{-2} K^{2} H$. Then
\begin{align}
\nonumber U^{2} V^{2} W^{2} F \ll&\ U^{2} V^{2} W^{2} \min \left\{V^{-4} T, V^{-12} M^{2} T, W^{-10} H^{5}, U^{-4} W^{-2} K^{2} H\right\} \\
\nonumber \leqslant&\ U^{2} V^{2} W^{2}\left(V^{-4} T\right)^{\frac{7}{20}}\left(V^{-12} M^{2} T\right)^{\frac{1}{20}}\left(W^{-10} H^{5}\right)^{\frac{1}{10}}\left(U^{-4} W^{-2} K^{2} H\right)^{\frac{1}{2}} \\
\nonumber =&\ T^{\frac{2}{5}} M^{\frac{1}{10}} H K \\
\nonumber =&\ T^{\frac{2}{5}} (M H K) M^{-\frac{9}{10}} \\
\nonumber \ll&\ x^{1-\varepsilon^2},
\end{align}
where $X^{\frac{164}{387}} \ll M$ is required.

(f) $F>2 V^{-4} M^{2}, F \leqslant 2 W^{-10} H^{5}, F>2 U^{-4} W^{-2} K^{2} H$. Then
\begin{align}
\nonumber U^{2} V^{2} W^{2} F \ll&\ U^{2} V^{2} W^{2} \min \left\{V^{-4} T, V^{-12} M^{2} T, W^{-10} H^{5}, U^{-4} W^{-2} T, U^{-12} W^{-6} K^{2} H T\right\} \\
\nonumber \leqslant&\ U^{2} V^{2} W^{2}\left(V^{-4} T\right)^{\frac{7}{20}}\left(V^{-12} M^{2} T\right)^{\frac{1}{20}}\left(W^{-10} H^{5}\right)^{\frac{1}{10}}\left(U^{-4} W^{-2} T\right)^{\frac{1}{2}} \\
\nonumber =&\ T^{\frac{9}{10}} \left( M^{\frac{1}{10}} H^{\frac{1}{2}} \right) \\
\nonumber \ll&\ x^{1-\varepsilon^2},
\end{align}
where $M H^{5} \ll X^{\frac{61}{43}}$ is required. This can be obtained by using $M H^{5} \ll X^{\frac{60}{43}}$.

(g) $F>2 V^{-4} M^{2}, F> 2 W^{-10} H^{5}, F \leqslant 2 U^{-4} W^{-2} K^{2} H$. Then
\begin{align}
\nonumber U^{2} V^{2} W^{2} F \ll&\ U^{2} V^{2} W^{2} \min \left\{V^{-4} T, V^{-12} M^{2} T, W^{-10} T, W^{-30} H^{5} T, U^{-4} W^{-2} K^{2} H\right\} \\
\nonumber \leqslant&\ U^{2} V^{2} W^{2}\left(V^{-4} T\right)^{\frac{9}{20}}\left(V^{-12} M^{2} T\right)^{\frac{1}{60}}\left(W^{-30} H^{5} T\right)^{\frac{1}{30}}\left(U^{-4} W^{-2} K^{2} H\right)^{\frac{1}{2}} \\
\nonumber =&\ T^{\frac{1}{2}} M^{\frac{1}{30}} H^{\frac{2}{3}} K \\
\nonumber =&\ T^{\frac{1}{2}} (M H K) \left( M^{-\frac{29}{30}} H^{-\frac{1}{3}}\right) \\
\nonumber \ll&\ x^{1-\varepsilon^2},
\end{align}
where $X^{\frac{615}{43}} \ll M^{29} H^{10}$ is required.

(h) $F>2 V^{-4} M^{2}, F> 2 W^{-10} H^{5}, F> 2 U^{-4} W^{-2} K^{2} H$. Then
\begin{align}
\nonumber U^{2} V^{2} W^{2} F \ll&\ U^{2} V^{2} W^{2} \min \left\{V^{-4}, V^{-12} M^{2}, W^{-10}, W^{-30} H^{5}, U^{-4} W^{-2}, U^{-12} W^{-6} K^{2} H\right\} T \\
\nonumber \leqslant&\ U^{2} V^{2} W^{2}\left(V^{-4}\right)^{\frac{9}{20}}\left(V^{-12} M^{2}\right)^{\frac{1}{60}}\left(W^{-30} H^{5}\right)^{\frac{1}{30}}\left(U^{-4} W^{-2}\right)^{\frac{1}{2}} T \\
\nonumber =&\ T \left( M^{\frac{1}{30}} H^{\frac{1}{6}}\right) \\
\nonumber \ll&\ x^{1-\varepsilon^2},
\end{align}
where $M H^{5} \ll X^{\frac{60}{43}}$ is required.

Assuming the condition (6). By applying the mean value estimate and Halász method to $M^{2}(s), H^{4}(s)$ and $K^{3}(s)$, we get
$$
I \ll U^{2} V^{2} W^{2} x^{-1} F (\log x)^{c}
$$
where
\begin{align}
\nonumber F = \min & \left\{ V^{-4}\left(M^{2}+T\right), V^{-4} M^{2}+V^{-12} M^{2} T, W^{-8}\left(H^{4}+T\right), \right. \\
\nonumber & \left. W^{-8} H^{4}+W^{-24} H^{4} T, U^{-6}\left(K^{3}+T\right), U^{-6} K^{3}+U^{-18} K^{3} T\right\}.
\end{align}
We consider 8 cases:

(a) $F \leqslant 2 V^{-4} M^{2}, F \leqslant 2 W^{-8} H^{4}, F \leqslant 2 U^{-6} K^{3}$. Then
\begin{align}
\nonumber U^{2} V^{2} W^{2} F \ll&\ U^{2} V^{2} W^{2} \min \left\{V^{-4} M^{2}, W^{-8} H^{4}, U^{-6} K^{3}\right\} \\
\nonumber \leqslant&\ U^{2} V^{2} W^{2}\left(V^{-4} M^{2}\right)^{\frac{1}{2}}\left(W^{-8} H^{4}\right)^{\frac{1}{6}}\left(U^{-6} K^{3}\right)^{\frac{1}{3}} \\
\nonumber =&\ W^{\frac{2}{3}} M H^{\frac{2}{3}} K \\
\nonumber \ll&\ x (\log x)^{-11 B}.
\end{align}

(b) $F \leqslant 2 V^{-4} M^{2}, F \leqslant 2 W^{-8} H^{4}, F>2 U^{-6} K^{3}$. Then
\begin{align}
\nonumber U^{2} V^{2} W^{2} F \ll&\ U^{2} V^{2} W^{2} \min \left\{V^{-4} M^{2}, W^{-8} H^{4}, U^{-6} T, U^{-18} K^{3} T\right\} \\
\nonumber \leqslant&\ U^{2} V^{2} W^{2}\left(V^{-4} M^{2}\right)^{\frac{1}{2}}\left(W^{-8} H^{4}\right)^{\frac{1}{4}}\left(U^{-6} T\right)^{\frac{5}{24}}\left(U^{-18} K^{3} T\right)^{\frac{1}{24}} \\
\nonumber =&\ T^{\frac{1}{4}} M H K^{\frac{1}{8}} \\
\nonumber =&\ T^{\frac{1}{4}} (M H K)^{\frac{1}{8}} \left(M^{\frac{7}{8}} H^{\frac{7}{8}} \right) \\
\nonumber \ll&\ x^{1-\varepsilon^2},
\end{align}
where $M H \ll X^{\frac{219}{301}}$ is required.

(c) $F \leqslant 2 V^{-4} M^{2}, F>2 W^{-8} H^{4}, F \leqslant 2 U^{-6} K^{3}$. Then
\begin{align}
\nonumber U^{2} V^{2} W^{2} F \ll&\ U^{2} V^{2} W^{2} \min \left\{V^{-4} M^{2}, W^{-8} T, W^{-24} H^{4} T, U^{-6} K^{3}\right\} \\
\nonumber \leqslant&\ U^{2} V^{2} W^{2}\left(V^{-4} M^{2}\right)^{\frac{1}{2}}\left(W^{-8} T\right)^{\frac{1}{8}}\left(W^{-24} H^{4} T\right)^{\frac{1}{24}}\left(U^{-6} K^{3}\right)^{\frac{1}{3}} \\
\nonumber =&\ T^{\frac{1}{6}} M H^{\frac{1}{6}} K \\
\nonumber =&\ T^{\frac{1}{6}} (M H K) H^{-\frac{5}{6}} \\
\nonumber \ll&\ x^{1-\varepsilon^2},
\end{align}
where $X^{\frac{41}{215}} \ll H$ is required.

(d) $F \leqslant 2 V^{-4} M^{2}, F>2 W^{-8} H^{4}, F> 2 U^{-6} K^{3}$. Then
\begin{align}
\nonumber U^{2} V^{2} W^{2} F \ll&\ U^{2} V^{2} W^{2} \min \left\{V^{-4} M^{2}, W^{-8} T, W^{-24} H^{4} T, U^{-6} T, U^{-18} K^{3} T\right\} \\
\nonumber \leqslant&\ U^{2} V^{2} W^{2}\left(V^{-4} M^{2}\right)^{\frac{1}{2}}\left(W^{-8} T\right)^{\frac{1}{8}}\left(W^{-24} H^{4} T\right)^{\frac{1}{24}}\left(U^{-6} T\right)^{\frac{1}{3}} \\
\nonumber =&\ T^{\frac{1}{2}} \left( M H^{\frac{1}{6}}\right) \\ 
\nonumber \ll&\ x^{1-\varepsilon^2},
\end{align}
where $M^{6} H \ll X^{\frac{135}{43}}$ is required.

(e) $F>2 V^{-4} M^{2}, F \leqslant 2 W^{-8} H^{4}, F \leqslant 2 U^{-6} K^{3}$. Then
\begin{align}
\nonumber U^{2} V^{2} W^{2} F \ll&\ U^{2} V^{2} W^{2} \min \left\{V^{-4} T, V^{-12} M^{2} T, W^{-8} H^{4}, U^{-6} K^{3}\right\} \\
\nonumber \leqslant&\ U^{2} V^{2} W^{2}\left(V^{-4} T\right)^{\frac{3}{8}}\left(V^{-12} M^{2} T\right)^{\frac{1}{24}}\left(W^{-8} H^{4}\right)^{\frac{1}{4}}\left(U^{-6} K^{3}\right)^{\frac{1}{3}} \\
\nonumber =&\ T^{\frac{5}{12}} M^{\frac{1}{12}} H K \\
\nonumber =&\ T^{\frac{5}{12}} (M H K) M^{-\frac{11}{12}} \\
\nonumber \ll&\ x^{1-\varepsilon^2},
\end{align}
where $X^{\frac{205}{473}} \ll M$ is required.

(f) $F>2 V^{-4} M^{2}, F \leqslant 2 W^{-8} H^{4}, F>2 U^{-6} K^{3}$. Then
\begin{align}
\nonumber U^{2} V^{2} W^{2} F \ll&\ U^{2} V^{2} W^{2} \min \left\{V^{-4} T, V^{-12} M^{2} T, W^{-8} H^{4}, U^{-6} T, U^{-18} K^{3} T\right\} \\
\nonumber \leqslant&\ U^{2} V^{2} W^{2}\left(V^{-4} T\right)^{\frac{1}{2}}\left(W^{-8} H^{4}\right)^{\frac{1}{4}}\left(U^{-6} T\right)^{\frac{5}{24}}\left(U^{-18} K^{3} T\right)^{\frac{1}{24}} \\
\nonumber =&\ T^{\frac{3}{4}} H K^{\frac{1}{8}} \\
\nonumber =&\ T^{\frac{3}{4}} (M H K)^{\frac{1}{8}} \left(M^{-\frac{1}{8}} H^{\frac{7}{8}} \right) \\
\nonumber \ll&\ x^{1-\varepsilon^2},
\end{align}
where $M^{-1} H^{7} \ll X^{\frac{55}{43}}$ is required.

(g) $F>2 V^{-4} M^{2}, F> 2 W^{-8} H^{4}, F \leqslant 2 U^{-6} K^{3}$. Then
\begin{align}
\nonumber U^{2} V^{2} W^{2} F \ll&\ U^{2} V^{2} W^{2} \min \left\{V^{-4} T, V^{-12} M^{2} T, W^{-8} T, W^{-24} H^{4} T, U^{-6} K^{3}\right\} \\
\nonumber \leqslant&\ U^{2} V^{2} W^{2}\left(V^{-4} T\right)^{\frac{1}{2}}\left(W^{-8} T\right)^{\frac{1}{8}}\left(W^{-24} H^{4} T\right)^{\frac{1}{24}}\left(U^{-6} K^{3}\right)^{\frac{1}{3}} \\
\nonumber =&\ T^{\frac{2}{3}} H^{\frac{1}{6}} K \\
\nonumber =&\ T^{\frac{2}{3}} (M H K) \left( M^{-1} H^{-\frac{5}{6}} \right) \\
\nonumber \ll&\ x^{1-\varepsilon^2},
\end{align}
where $X^{\frac{164}{43}} \ll M^{6} H^{5}$ is required.

(h) $F>2 V^{-4} M^{2}, F> 2 W^{-8} H^{4}, F> 2 U^{-6} K^{3}$. Then
\begin{align}
\nonumber U^{2} V^{2} W^{2} F \ll&\ U^{2} V^{2} W^{2} \min \left\{V^{-4}, V^{-12} M^{2}, W^{-8}, W^{-24} H^{4}, U^{-6}, U^{-18} K^{3}\right\} T \\
\nonumber \leqslant&\ U^{2} V^{2} W^{2}\left(V^{-4}\right)^{\frac{1}{2}}\left(W^{-8}\right)^{\frac{1}{4}}\left(U^{-6}\right)^{\frac{5}{24}}\left(U^{-18} K^{3}\right)^{\frac{1}{24}} T \\
\nonumber =&\ T K^{\frac{1}{8}} \\
\nonumber =&\ T (M H K)^{\frac{1}{8}} \left(M^{-\frac{1}{8}} H^{-\frac{1}{8}} \right) \\
\nonumber \ll&\ x^{1-\varepsilon^2},
\end{align}
where $X^{\frac{27}{43}} \ll M H$ is required.

Assuming the condition (7). By applying the mean value estimate and Halász method to $M^{2}(s) H(s), H^{6}(s)$ and $K^{2}(s)$, we get
$$
I \ll U^{2} V^{2} W^{2} x^{-1} F (\log x)^{c}
$$
where
\begin{align}
\nonumber F = \min & \left\{ V^{-4} W^{-2}\left(M^{2} H+T\right), V^{-4} W^{-2} M^{2} H+V^{-12} W^{-6} M^{2} H T, \right. \\
\nonumber & \left. W^{-12}\left(H^{6}+T\right), W^{-12} H^{6}+W^{-36} H^{6} T, U^{-4}\left(K^{2}+T\right), U^{-4} K^{2}+U^{-12} K^{2} T\right\}.
\end{align}
We consider 8 cases:

(a) $F \leqslant 2 V^{-4} W^{-2} M^{2} H, F \leqslant 2 W^{-12} H^{6}, F \leqslant 2 U^{-4} K^{2}$. Then
\begin{align}
\nonumber U^{2} V^{2} W^{2} F \ll&\ U^{2} V^{2} W^{2} \min \left\{V^{-4} W^{-2} M^{2} H, W^{-12} H^{6}, U^{-4} K^{2}\right\} \\
\nonumber \leqslant&\ U^{2} V^{2} W^{2}\left(V^{-4} W^{-2} M^{2} H\right)^{\frac{1}{2}}\left(U^{-4} K^{2}\right)^{\frac{1}{2}} \\
\nonumber =&\ W M H^{\frac{1}{2}} K \\
\nonumber \ll&\ x (\log x)^{-11 B}.
\end{align}

(b) $F \leqslant 2 V^{-4} W^{-2} M^{2} H, F \leqslant 2 W^{-12} H^{6}, F>2 U^{-4} K^{2}$. Then
\begin{align}
\nonumber U^{2} V^{2} W^{2} F \ll&\ U^{2} V^{2} W^{2} \min \left\{V^{-4} W^{-2} M^{2} H, W^{-12} H^{6}, U^{-4} T, U^{-12} K^{2} T\right\} \\
\nonumber \leqslant&\ U^{2} V^{2} W^{2}\left(V^{-4} W^{-2} M^{2} H\right)^{\frac{1}{2}}\left(W^{-12} H^{6}\right)^{\frac{1}{12}}\left(U^{-4} T\right)^{\frac{3}{8}}\left(U^{-12} K^{2} T\right)^{\frac{1}{24}} \\
\nonumber =&\ T^{\frac{5}{12}} M H K^{\frac{1}{12}} \\
\nonumber =&\ T^{\frac{5}{12}} (M H K)^{\frac{1}{12}} \left(M^{\frac{11}{12}} H^{\frac{11}{12}} \right) \\
\nonumber \ll&\ x^{1-\varepsilon^2},
\end{align}
where $M H \ll X^{\frac{268}{473}}$ is required.

(c) $F \leqslant 2 V^{-4} W^{-2} M^{2} H, F>2 W^{-12} H^{6}, F \leqslant 2 U^{-4} K^{2}$. Then
\begin{align}
\nonumber U^{2} V^{2} W^{2} F \ll&\ U^{2} V^{2} W^{2} \min \left\{V^{-4} W^{-2} M^{2} H, W^{-12} T, W^{-36} H^{6} T, U^{-4} K^{2}\right\} \\
\nonumber \leqslant&\ U^{2} V^{2} W^{2}\left(V^{-4} W^{-2} M^{2} H\right)^{\frac{1}{2}}\left(U^{-4} K^{2}\right)^{\frac{1}{2}} \\
\nonumber =&\ W M H^{\frac{1}{2}} K \\
\nonumber \ll&\ x (\log x)^{-11 B}.
\end{align}

(d) $F \leqslant 2 V^{-4} W^{-2} M^{2} H, F>2 W^{-12} H^{6}, F> 2 U^{-4} K^{2}$. Then
\begin{align}
\nonumber U^{2} V^{2} W^{2} F \ll&\ U^{2} V^{2} W^{2} \min \left\{V^{-4} W^{-2} M^{2} H, W^{-12} T, W^{-36} H^{6} T, U^{-4} T, U^{-12} K^{2} T\right\} \\
\nonumber \leqslant&\ U^{2} V^{2} W^{2}\left(V^{-4} W^{-2} M^{2} H\right)^{\frac{1}{2}}\left(W^{-36} H^{6} T\right)^{\frac{1}{36}}\left(U^{-4} T\right)^{\frac{11}{24}}\left(U^{-12} K^{2} T\right)^{\frac{1}{72}} \\
\nonumber =&\ T^{\frac{1}{2}} M H^{\frac{2}{3}} K^{\frac{1}{36}} \\
\nonumber =&\ T^{\frac{1}{2}} (M H K)^{\frac{1}{36}} \left( M^{\frac{35}{36}} H^{\frac{23}{36}} \right) \\
\nonumber \ll&\ x^{1-\varepsilon^2},
\end{align}
where $M^{35} H^{23} \ll X^{\frac{767}{43}}$ is required.

(e) $F>2 V^{-4} W^{-2} M^{2} H, F \leqslant 2 W^{-12} H^{6}, F \leqslant 2 U^{-4} K^{2}$. Then
\begin{align}
\nonumber U^{2} V^{2} W^{2} F \ll&\ U^{2} V^{2} W^{2} \min \left\{V^{-4} W^{-2} T, V^{-12} W^{-6} M^{2} H T, W^{-12} H^{6}, U^{-4} K^{2}\right\} \\
\nonumber \leqslant&\ U^{2} V^{2} W^{2}\left(V^{-4} W^{-2} T\right)^{\frac{3}{8}}\left(V^{-12} W^{-6} M^{2} H T\right)^{\frac{1}{24}}\left(W^{-12} H^{6}\right)^{\frac{1}{12}}\left(U^{-4} K^{2}\right)^{\frac{1}{2}} \\
\nonumber =&\ T^{\frac{5}{12}} M^{\frac{1}{12}} H^{\frac{13}{24}} K \\
\nonumber =&\ T^{\frac{5}{12}} (M H K) \left( M^{-\frac{11}{12}} H^{-\frac{11}{24}} \right) \\
\nonumber \ll&\ x^{1-\varepsilon^2},
\end{align}
where $X^{\frac{410}{473}} \ll M^{2} H$ is required.

(f) $F>2 V^{-4} W^{-2} M^{2} H, F \leqslant 2 W^{-12} H^{6}, F>2 U^{-4} K^{2}$. Then
\begin{align}
\nonumber U^{2} V^{2} W^{2} F \ll&\ U^{2} V^{2} W^{2} \min \left\{V^{-4} W^{-2} T, V^{-12} W^{-6} M^{2} H T, W^{-12} H^{6}, U^{-4} T, U^{-12} K^{2} T\right\} \\
\nonumber \leqslant&\ U^{2} V^{2} W^{2}\left(V^{-4} W^{-2} T\right)^{\frac{3}{8}}\left(V^{-12} W^{-6} M^{2} H T\right)^{\frac{1}{24}}\left(W^{-12} H^{6}\right)^{\frac{1}{12}}\left(U^{-4} T\right)^{\frac{1}{2}} \\
\nonumber =&\ T^{\frac{11}{12}} \left( M^{\frac{1}{12}} H^{\frac{13}{24}} \right) \\
\nonumber \ll&\ x^{1-\varepsilon^2},
\end{align}
where $M^{2} H^{13} \ll X^{\frac{130}{43}}$ is required.

(g) $F>2 V^{-4} W^{-2} M^{2} H, F> 2 W^{-12} H^{6}, F \leqslant 2 U^{-4} K^{2}$. Then
\begin{align}
\nonumber U^{2} V^{2} W^{2} F \ll&\ U^{2} V^{2} W^{2} \min \left\{V^{-4} W^{-2} T, V^{-12} W^{-6} M^{2} H T, W^{-12} T, W^{-36} H^{6} T, U^{-4} K^{2}\right\} \\
\nonumber \leqslant&\ U^{2} V^{2} W^{2}\left(V^{-4} W^{-2} T\right)^{\frac{11}{24}}\left(V^{-12} W^{-6} M^{2} H T\right)^{\frac{1}{7_{2}}}\left(W^{-36} H^{6} T\right)^{\frac{1}{36}}\left(U^{-4} K^{2}\right)^{\frac{1}{2}} \\
\nonumber =&\ T^{\frac{1}{2}} M^{\frac{1}{36}} H^{\frac{13}{72}} K \\
\nonumber =&\ T^{\frac{1}{2}} (M H K) \left( M^{-\frac{35}{36}} H^{-\frac{59}{72}} \right) \\
\nonumber \ll&\ x^{1-\varepsilon^2},
\end{align}
where $X^{\frac{1476}{43}} \ll M^{70} H^{59}$ is required.

(h) $F>2 V^{-4} W^{-2} M^{2} H, F> 2 W^{-12} H^{6}, F> 2 U^{-4} K^{2}$. Then
\begin{align}
\nonumber U^{2} V^{2} W^{2} F \ll&\ U^{2} V^{2} W^{2} \min \left\{V^{-4} W^{-2}, V^{-12} W^{-6} M^{2} H, W^{-12}, W^{-36} H^{6}, U^{-4}, U^{-12} K^{2}\right\} T \\
\nonumber \leqslant&\ U^{2} V^{2} W^{2}\left(V^{-4} W^{-2}\right)^{\frac{1}{2}}\left(W^{-12}\right)^{\frac{1}{12}}\left(U^{-4}\right)^{\frac{3}{8}}\left(U^{-12} K^{2}\right)^{\frac{1}{24}} T \\
\nonumber =&\ T K^{\frac{1}{12}} \\
\nonumber =&\ T (M H K)^{\frac{1}{12}} \left(M^{-\frac{1}{12}} H^{-\frac{1}{12}} \right) \\
\nonumber \ll&\ x^{1-\varepsilon^2},
\end{align}
where $X^{\frac{19}{43}} \ll M H$ is required. This can be obtained by using $X^{\frac{738}{1505}} \ll M H^{\frac{59}{70}} \ll M H$.

Assuming the condition (8). By applying the mean value estimate and Halász method to $M^{2}(s) H(s), H^{7}(s)$ and $K^{2}(s)$, we get
$$
I \ll U^{2} V^{2} W^{2} x^{-1} F (\log x)^{c}
$$
where
\begin{align}
\nonumber F = \min & \left\{ V^{-4} W^{-2}\left(M^{2} H+T\right), V^{-4} W^{-2} M^{2} H+V^{-12} W^{-6} M^{2} H T, \right. \\
\nonumber & \left. W^{-14}\left(H^{7}+T\right), W^{-14} H^{7}+W^{-42} H^{7} T, U^{-4}\left(K^{2}+T\right) U^{-4} K^{2}+U^{-12} K^{2} T\right\}.
\end{align}
We consider 8 cases:

(a) $F \leqslant 2 V^{-4} W^{-2} M^{2} H, F \leqslant 2 W^{-14} H^{7}, F \leqslant 2 U^{-4} K^{2}$. Then
\begin{align}
\nonumber U^{2} V^{2} W^{2} F \ll&\ U^{2} V^{2} W^{2} \min \left\{V^{-4} W^{-2} M^{2} H, W^{-14} H^{7}, U^{-4} K^{2}\right\} \\
\nonumber \leqslant&\ U^{2} V^{2} W^{2}\left(V^{-4} W^{-2} M^{2} H\right)^{\frac{1}{2}}\left(U^{-4} K^{2}\right)^{\frac{1}{2}} \\
\nonumber =&\ W M H^{\frac{1}{2}} K \\
\nonumber \ll&\ x (\log x)^{-11 B}.
\end{align}

(b) $F \leqslant 2 V^{-4} W^{-2} M^{2} H, F \leqslant 2 W^{-14} H^{7}, F>2 U^{-4} K^{2}$. Then
\begin{align}
\nonumber U^{2} V^{2} W^{2} F \ll&\ U^{2} V^{2} W^{2} \min \left\{V^{-4} W^{-2} M^{2} H, W^{-14} H^{7}, U^{-4} T, U^{-12} K^{2} T\right\} \\
\nonumber \leqslant&\ U^{2} V^{2} W^{2}\left(V^{-4} W^{-2} M^{2} H\right)^{\frac{1}{2}}\left(W^{-14} H^{7}\right)^{\frac{1}{14}}\left(U^{-4} T\right)^{\frac{11}{28}}\left(U^{-12} K^{2} T\right)^{\frac{1}{28}} \\
\nonumber =&\ T^{\frac{3}{7}} M H K^{\frac{1}{14}} \\
\nonumber =&\ T^{\frac{3}{7}} (M H K)^{\frac{1}{14}} \left(M^{\frac{13}{14}} H^{\frac{13}{14}} \right) \\
\nonumber \ll&\ x^{1-\varepsilon^2},
\end{align}
where $M H \ll X^{\frac{313}{559}}$ is required.

(c) $F \leqslant 2 V^{-4} W^{-2} M^{2} H, F>2 W^{-14} H^{7}, F \leqslant 2 U^{-4} K^{2}$. Then 
\begin{align}
\nonumber U^{2} V^{2} W^{2} F \ll&\ U^{2} V^{2} W^{2} \min \left\{V^{-4} W^{-2} M^{2} H, W^{-14} T, W^{-42} H^{7} T, U^{-4} K^{2}\right\} \\
\nonumber \leqslant&\ U^{2} V^{2} W^{2}\left(V^{-4} W^{-2} M^{2} H\right)^{\frac{1}{2}}\left(U^{-4} K^{2}\right)^{\frac{1}{2}} \\
\nonumber =&\ W M H^{\frac{1}{2}} K \\
\nonumber \ll&\ x (\log x)^{-11 B}.
\end{align}

(d) $F \leqslant 2 V^{-4} W^{-2} M^{2} H, F>2 W^{-14} H^{7}, F> 2 U^{-4} K^{2}$. Then
\begin{align}
\nonumber U^{2} V^{2} W^{2} F \ll&\ U^{2} V^{2} W^{2} \min \left\{V^{-4} W^{-2} M^{2} H, W^{-14} T, W^{-42} H^{7} T, U^{-4} T, U^{-12} K^{2} T\right\} \\
\nonumber \leqslant&\ U^{2} V^{2} W^{2}\left(V^{-4} W^{-2} M^{2} H\right)^{\frac{1}{2}}\left(W^{-42} H^{7} T\right)^{\frac{1}{42}}\left(U^{-4} T\right)^{\frac{13}{28}}\left(U^{-12} K^{2} T\right)^{\frac{1}{84}} \\
\nonumber =&\ T^{\frac{1}{2}} M H^{\frac{2}{3}} K^{\frac{1}{42}} \\
\nonumber =&\ T^{\frac{1}{2}} (M H K)^{\frac{1}{42}} \left( M^{\frac{41}{42}} H^{\frac{9}{14}} \right) \\
\nonumber \ll&\ x^{1-\varepsilon^2},
\end{align}
where $M^{41} H^{27} \ll X^{\frac{902}{43}}$ is required.

(e) $F>2 V^{-4} W^{-2} M^{2} H, F \leqslant 2 W^{-14} H^{7}, F \leqslant 2 U^{-4} K^{2}$. Then 
\begin{align}
\nonumber U^{2} V^{2} W^{2} F \ll&\ U^{2} V^{2} W^{2} \min \left\{V^{-4} W^{-2} T, V^{-12} W^{-6} M^{2} H T, W^{-14} H^{7}, U^{-4} K^{2}\right\} \\
\nonumber \leqslant&\ U^{2} V^{2} W^{2}\left(V^{-4} W^{-2} T\right)^{\frac{11}{28}}\left(V^{-12} W^{-6} M^{2} H T\right)^{\frac{1}{28}}\left(W^{-14} H^{7}\right)^{\frac{1}{14}}\left(U^{-4} K^{2}\right)^{\frac{1}{2}} \\
\nonumber =&\ T^{\frac{3}{7}} M^{\frac{1}{14}} H^{\frac{15}{28}} K \\
\nonumber =&\ T^{\frac{3}{7}} (M H K) \left( M^{-\frac{13}{14}} H^{-\frac{13}{28}} \right) \\
\nonumber \ll&\ x^{1-\varepsilon^2},
\end{align}
where $X^{\frac{492}{559}} \ll M^{2} H$ is required.

(f) $F>2 V^{-4} W^{-2} M^{2} H, F \leqslant 2 W^{-14} H^{7}, F>2 U^{-4} K^{2}$. Then 
\begin{align}
\nonumber U^{2} V^{2} W^{2} F \ll&\ U^{2} V^{2} W^{2} \min \left\{V^{-4} W^{-2} T, V^{-12} W^{-6} M^{2} H T, W^{-14} H^{7}, U^{-4} T, U^{-12} K^{2} T\right\} \\
\nonumber \leqslant&\ U^{2} V^{2} W^{2}\left(V^{-4} W^{-2} T\right)^{\frac{11}{28}}\left(V^{-12} W^{-6} M^{2} H T\right)^{\frac{1}{28}}\left(W^{-14} H^{7}\right)^{\frac{1}{14}}\left(U^{-4} T\right)^{\frac{1}{2}} \\
\nonumber =&\ T^{\frac{13}{14}} \left( M^{\frac{1}{14}} H^{\frac{15}{28}} \right) \\
\nonumber \ll&\ x^{1-\varepsilon^2},
\end{align}
where $M^{2} H^{15} \ll X^{\frac{138}{43}}$ is required.

(g) $F>2 V^{-4} W^{-2} M^{2} H, F> 2 W^{-14} H^{7}, F \leqslant 2 U^{-4} K^{2}$. Then
\begin{align}
\nonumber U^{2} V^{2} W^{2} F \ll&\ U^{2} V^{2} W^{2} \min \left\{V^{-4} W^{-2} T, V^{-12} W^{-6} M^{2} H T, W^{-14} T, W^{-42} H^{7} T, U^{-4} K^{2}\right\} \\
\nonumber \leqslant&\ U^{2} V^{2} W^{2}\left(V^{-4} W^{-2} T\right)^{\frac{13}{88}}\left(V^{-12} W^{-6} M^{2} H T\right)^{\frac{1}{84}} \left(W^{-42} H^{7} T\right)^{\frac{1}{42}}\left(U^{-4} K^{2}\right)^{\frac{1}{2}} \\
\nonumber =&\ T^{\frac{1}{2}} M^{\frac{1}{42}} H^{\frac{5}{28}} K \\
\nonumber =&\ T^{\frac{1}{2}} (M H K) \left( M^{-\frac{41}{42}} H^{-\frac{23}{28}} \right) \\
\nonumber \ll&\ x^{1-\varepsilon^2},
\end{align}
where $X^{\frac{1722}{43}} \ll M^{82} H^{69}$ is required.

(h) $F>2 V^{-4} W^{-2} M^{2} H, F> 2 W^{-14} H^{7}, F> 2 U^{-4} K^{2}$. Then
\begin{align}
\nonumber U^{2} V^{2} W^{2} F \ll&\ U^{2} V^{2} W^{2} \min \left\{V^{-4} W^{-2}, V^{-12} W^{-6} M^{2} H, W^{-14}, W^{-42} H^{7}, U^{-4}, U^{-12} K^{2}\right\} T \\
\nonumber \leqslant&\ U^{2} V^{2} W^{2}\left(V^{-4} W^{-2}\right)^{\frac{1}{2}}\left(W^{-14}\right)^{\frac{1}{14}}\left(U^{-4}\right)^{\frac{11}{28}}\left(U^{-12} K^{2}\right)^{\frac{1}{28}} T \\
\nonumber =&\ T K^{\frac{1}{14}} \\
\nonumber =&\ T (M H K)^{\frac{1}{14}} \left(M^{-\frac{1}{14}} H^{-\frac{1}{14}} \right) \\
\nonumber \ll&\ x^{1-\varepsilon^2},
\end{align}
where $X^{\frac{15}{43}} \ll M H$ is required. This can be obtained by using $X^{\frac{246}{559}} \ll M H^{\frac{1}{2}} \ll M H$.

Assuming the condition (9). By applying the mean value estimate and Halász method to $M^{2}(s), H^{6}(s)$ and $K^{2}(s) H(s)$, we get
$$
I \ll U^{2} V^{2} W^{2} x^{-1} F (\log x)^{c}
$$
where
\begin{align}
\nonumber F = \min & \left\{ V^{-4}\left(M^{2}+T\right), V^{-4} M^{2}+V^{-12} M^{2} T, W^{-12}\left(H^{6}+T\right), W^{-12} H^{6} +W^{-36} H^{6} T, \right. \\
\nonumber & \left. U^{-4} W^{-2}\left(K^{2} H+T\right), U^{-4} W^{-2} K^{2} H +U^{-12} W^{-6} K^{2} H T\right\}.
\end{align}
We consider 8 cases:

(a) $F \leqslant 2 V^{-4} M^{2}, F \leqslant 2 W^{-12} H^{6}, F \leqslant 2 U^{-4} W^{-2} K^{2} H$. Then
\begin{align}
\nonumber U^{2} V^{2} W^{2} F \ll&\ U^{2} V^{2} W^{2} \min \left\{V^{-4} M^{2}, W^{-12} H^{6}, U^{-4} W^{-2} K^{2} H\right\} \\
\nonumber \leqslant&\ U^{2} V^{2} W^{2}\left(V^{-4} M^{2}\right)^{\frac{1}{2}}\left(U^{-4} W^{-2} K^{2} H\right)^{\frac{1}{2}} \\
\nonumber =&\ W M H^{\frac{1}{2}} K \\
\nonumber \ll&\ x (\log x)^{-11 B}.
\end{align}

(b) $F \leqslant 2 V^{-4} M^{2}, F \leqslant 2 W^{-12} H^{6}, F>2 U^{-4} W^{-2} K^{2} H$. Then
\begin{align}
\nonumber U^{2} V^{2} W^{2} F \ll&\ U^{2} V^{2} W^{2} \min \left\{V^{-4} M^{2}, W^{-12} H^{6}, U^{-4} W^{-2} T, U^{-12} W^{-6} K^{2} H T\right\} \\
\nonumber \leqslant&\ U^{2} V^{2} W^{2}\left(V^{-4} M^{2}\right)^{\frac{1}{2}}\left(W^{-12} H^{6}\right)^{\frac{1}{12}}\left(U^{-4} W^{-2} T\right)^{\frac{3}{8}}\left(U^{-12} W^{-6} K^{2} H T\right)^{\frac{1}{24}} \\
\nonumber =&\ T^{\frac{5}{12}} M H^{\frac{13}{24}} K^{\frac{1}{12}} \\
\nonumber =&\ T^{\frac{5}{12}} (M H K)^{\frac{1}{12}} \left( M^{\frac{11}{12}} H^{\frac{11}{24}} \right) \\
\nonumber \ll&\ x^{1-\varepsilon^2},
\end{align}
where $M^{2} H \ll X^{\frac{536}{473}}$ is required.

(c) $F \leqslant 2 V^{-4} M^{2}, F>2 W^{-12} H^{6}, F \leqslant 2 U^{-4} W^{-2} K^{2} H$. Then
\begin{align}
\nonumber U^{2} V^{2} W^{2} F \ll&\ U^{2} V^{2} W^{2} \min \left\{V^{-4} M^{2}, W^{-12} T, W^{-36} H^{6} T, U^{-4} W^{-2} K^{2} H\right\} \\
\nonumber \leqslant&\ U^{2} V^{2} W^{2}\left(V^{-4} M^{2}\right)^{\frac{1}{2}}\left(U^{-4} W^{-2} K^{2} H\right)^{\frac{1}{2}} \\
\nonumber =&\ W M H^{\frac{1}{2}} K \\
\nonumber \ll&\ x (\log x)^{-11 B}.
\end{align}

(d) $F \leqslant 2 V^{-4} M^{2}, F>2 W^{-12} H^{6}, F> 2 U^{-4} W^{-2} K^{2} H$. Then
\begin{align}
\nonumber U^{2} V^{2} W^{2} F \ll&\ U^{2} V^{2} W^{2} \min \left\{V^{-4} M^{2}, W^{-12} T, W^{-36} H^{6} T, U^{-4} W^{-2} T, U^{-12} W^{-6} K^{2} H T\right\} \\
\nonumber \leqslant&\ U^{2} V^{2} W^{2}\left(V^{-4} M^{2}\right)^{\frac{1}{2}}\left(W^{-36} H^{6} T\right)^{\frac{1}{36}}\left(U^{-4} W^{-2} T\right)^{\frac{11}{24}} \left(U^{-12} W^{-6} K^{2} H T\right)^{\frac{1}{72}} \\
\nonumber =&\ T^{\frac{1}{2}} M H^{\frac{13}{72}} K^{\frac{1}{36}} \\
\nonumber =&\ T^{\frac{1}{2}} (M H K)^{\frac{1}{36}} \left( M^{\frac{35}{36}} H^{\frac{11}{72}} \right) \\
\nonumber \ll&\ x^{1-\varepsilon^2},
\end{align}
where $M^{70} H^{11} \ll X^{\frac{1534}{43}}$ is required.

(e) $F>2 V^{-4} M^{2}, F \leqslant 2 W^{-12} H^{6}, F \leqslant 2 U^{-4} W^{-2} K^{2} H$. Then
\begin{align}
\nonumber U^{2} V^{2} W^{2} F \ll&\ U^{2} V^{2} W^{2} \min \left\{V^{-4} T, V^{-12} M^{2} T, W^{-12} H^{6}, U^{-4} W^{-2} K^{2} H\right\} \\
\nonumber \leqslant&\ U^{2} V^{2} W^{2}\left(V^{-4} T\right)^{\frac{3}{8}}\left(V^{-12} M^{2} T\right)^{\frac{1}{24}}\left(W^{-12} H^{6}\right)^{\frac{1}{12}}\left(U^{-4} W^{-2} K^{2} H\right)^{\frac{1}{2}} \\
\nonumber =&\ T^{\frac{5}{12}} M^{\frac{1}{12}} H K \\
\nonumber =&\ T^{\frac{5}{12}} (M H K) M^{-\frac{11}{12}} \\
\nonumber \ll&\ x^{1-\varepsilon^2},
\end{align}
where $X^{\frac{205}{473}} \ll M$ is required.

(f) $F>2 V^{-4} M^{2}, F \leqslant 2 W^{-12} H^{6}, F>2 U^{-4} W^{-2} K^{2} H$. Then 
\begin{align}
\nonumber U^{2} V^{2} W^{2} F \ll&\ U^{2} V^{2} W^{2} \min \left\{V^{-4} T, V^{-12} M^{2} T, W^{-12} H^{6}, U^{-4} W^{-2} T, U^{-12} W^{-6} K^{2} H T\right\} \\
\nonumber \leqslant&\ U^{2} V^{2} W^{2}\left(V^{-4} T\right)^{\frac{3}{8}}\left(V^{-12} M^{2} T\right)^{\frac{1}{24}}\left(W^{-12} H^{6}\right)^{\frac{1}{12}}\left(U^{-4} W^{-2} T\right)^{\frac{1}{2}} \\
\nonumber =&\ T^{\frac{11}{12}} \left( M^{\frac{1}{12}} H^{\frac{1}{2}} \right) \\
\nonumber \ll&\ x^{1-\varepsilon^2},
\end{align}
where $M H^{6} \ll X^{\frac{65}{43}}$ is required.

(g) $F>2 V^{-4} M^{2}, F> 2 W^{-12} H^{6}, F \leqslant 2 U^{-4} W^{-2} K^{2} H$. Then
\begin{align}
\nonumber U^{2} V^{2} W^{2} F \ll&\ U^{2} V^{2} W^{2} \min \left\{V^{-4} T, V^{-12} M^{2} T, W^{-12} T, W^{-36} H^{6} T, U^{-4} W^{-2} K^{2} H\right\} \\
\nonumber \leqslant&\ U^{2} V^{2} W^{2}\left(V^{-4} T\right)^{\frac{11}{24}}\left(V^{-12} M^{2} T\right)^{\frac{1}{72}}\left(W^{-36} H^{6} T\right)^{\frac{1}{36}}\left(U^{-4} W^{-2} K^{2} H\right)^{\frac{1}{2}} \\
\nonumber =&\ T^{\frac{1}{2}} M^{\frac{1}{36}} H^{\frac{2}{3}} K \\
\nonumber =&\ T^{\frac{1}{2}} (M H K) \left( M^{-\frac{35}{36}} H^{-\frac{1}{3}} \right) \\
\nonumber \ll&\ x^{1-\varepsilon^2},
\end{align}
where $X^{\frac{738}{43}} \ll M^{35} H^{12}$ is required.

(h) $F>2 V^{-4} M^{2}, F> 2 W^{-12} H^{6}, F> 2 U^{-4} W^{-2} K^{2} H$. Then
\begin{align}
\nonumber U^{2} V^{2} W^{2} F \ll&\ U^{2} V^{2} W^{2} \min \left\{V^{-4}, V^{-12} M^{2}, W^{-12}, W^{-36} H^{6}, U^{-4} W^{-2}, U^{-12} W^{-6} K^{2} H\right\} T \\
\nonumber \leqslant&\ U^{2} V^{2} W^{2}\left(V^{-4}\right)^{\frac{3}{8}}\left(V^{-12} M^{2}\right)^{\frac{1}{24}}\left(W^{-12}\right)^{\frac{1}{12}}\left(U^{-4} W^{-2}\right)^{\frac{1}{2}} T \\
\nonumber =&\ T M^{\frac{1}{12}} \\
\nonumber \ll&\ x^{1-\varepsilon^2},
\end{align}
where $M \ll X^{\frac{24}{43}}$ is required. This can be obtained by using $M \ll M H^{\frac{11}{70}} \ll X^{\frac{767}{1505}}$.

Finally, by combining all the cases above, Lemma~\ref{l32} is proved.
\end{proof}

\begin{lemma}\label{l33}
Suppose that $N Q R K=X$, and $N$, $Q$ and $R$ satisfy one of the following 10 conditions:

(1) $X^{\frac{27267}{66994}} \ll N Q \ll X^{\frac{231}{559}},\ X^{\frac{79}{817}} \ll R \ll Q \ll (N Q)^{-\frac{35}{23}} X^{\frac{767}{989}} $;

(2) $X^{\frac{231}{559}} \ll N Q \ll X^{\frac{3275}{7826}},\ X^{\frac{79}{817}} \ll R \ll Q \ll (N Q)^{-1} X^{\frac{313}{559}} $;

(3) $X^{\frac{3275}{7826}} \ll N Q \ll X^{\frac{15005}{33497}},\ X^{\frac{79}{817}} \ll R \ll Q \ll (N Q)^{-\frac{41}{27}} X^{\frac{902}{1161}} $;

(4) $X^{\frac{13074}{28595}} \ll N Q \ll X^{\frac{20}{43}},\ X^{\frac{79}{817}} \ll R \ll Q \ll (N Q)^{-\frac{1}{5}} X^{\frac{12}{43}} $;

(5) $X^{\frac{20}{43}} \ll N Q \ll X^{\frac{41}{86}},\ X^{\frac{79}{817}} \ll R \ll Q \ll (N Q)^{\frac{1}{7}} X^{\frac{55}{301}} $;

(6) $X^{\frac{41}{86}} \ll N Q \ll X^{\frac{227}{473}},\ X^{\frac{79}{817}} \ll R \ll Q \ll (N Q)^{-1} X^{\frac{219}{301}} $;

(7) $X^{\frac{227}{473}} \ll N Q \ll X^{\frac{2333}{4859}},\ X^{\frac{79}{817}} \ll R \ll Q \ll (N Q)^{-\frac{58}{9}} X^{\frac{1264}{387}} $;

(8) $X^{\frac{2333}{4859}} \ll N Q \ll X^{\frac{2501}{5203}},\ X^{\frac{79}{817}} \ll R \ll Q \ll (N Q)^{-\frac{1}{6}} X^{\frac{65}{258}} $;

(9) $X^{\frac{2501}{5203}} \ll N Q \ll X^{\frac{499}{1032}},\ X^{\frac{79}{817}} \ll R \ll Q \ll (N Q)^{-2} X^{\frac{536}{473}} $;

(10) $X^{\frac{499}{1032}} \ll N Q \ll X^{\frac{28277}{57190}},\ X^{\frac{79}{817}} \ll R \ll Q \ll (N Q)^{-\frac{70}{11}} X^{\frac{1534}{473}} $.

Assume further that for $(\log X)^{B/ \varepsilon} \leqslant |t| \leqslant 2 X, N(b+i t) Q(b+i t) \ll (\log x)^{-B/ \varepsilon}$ and $R(b+i t) \ll (\log x)^{-B/ \varepsilon}$. Then for $(\log X)^{B/ \varepsilon} \leqslant T \leqslant X$, we have
$$
\left( \min \left(\eta, \frac{1}{T}\right) \right)^{2} \int_{T}^{2 T}|N(b+i t) Q(b+i t) R(b+i t) K(b+i t)|^{2} d t \ll \eta^{2} (\log x)^{-10 B}.
$$
\end{lemma}
\begin{proof}
The proof is similar to that of [\cite{Jia120}, Lemma 13]. Let $M(s)=N(s)Q(s)$ and $H(s)=R(s)$. Then by Lemma~\ref{l32}, Lemma~\ref{l33} is proved.
\end{proof}

\begin{lemma}\label{l34}
Suppose that $N Q R K=X$, and $N$, $Q$ and $R$ satisfy $X^{\frac{79}{817}} \ll R \ll Q$ and one of the following 9 conditions:

(1) $(N R) Q \ll X^{\frac{673}{1247}},\ X^{\frac{82}{473}} \ll Q,\ (N R)^{29} Q^{-1} \ll X^{\frac{427}{43}},\ X^{\frac{246}{817}} \ll (N R),\ (N R)^{-1} Q^{29} \ll X^{\frac{263}{43}},\ X^{\frac{246}{43}} \ll (N R)^{12} Q^{11}$;

(2) $(N R) Q \ll X^{\frac{223}{387}},\ (N R)^{29} Q^{19} \ll X^{\frac{632}{43}},\ X^{\frac{328}{387}} \ll (N R)^{2} Q,\ (N R)^{2} Q^{11} \ll X^{\frac{120}{43}},\ X^{\frac{1230}{43}} \ll (N R)^{58} Q^{49}$;

(3) $(N R) Q \ll X^{\frac{268}{473}},\ X^{\frac{41}{215}} \ll Q,\ (N R)^{6} Q \ll X^{\frac{94}{43}},\ X^{\frac{82}{301}} \ll (N R) \ll X^{\frac{16}{43}},\ (N R) Q^{8} \ll X^{\frac{98}{43}},\ X^{\frac{123}{43}} \ll (N R)^{6} Q^{5}$;

(4) $(N R) Q \ll X^{\frac{571}{817}},\ X^{\frac{82}{473}} \ll Q,\ (N R)^{12} Q \ll X^{\frac{270}{43}},\ X^{\frac{574}{1247}} \ll (N R),\ (N R)^{-1} Q^{19} \ll X^{\frac{161}{43}},\ X^{\frac{328}{43}} \ll (N R)^{12} Q^{11}$;

(5) $(N R)^{2} Q \ll X^{\frac{446}{387}},\ (N R)^{58} Q^{9} \ll X^{\frac{1264}{43}},\ X^{\frac{164}{387}} \ll (N R),\ (N R) Q^{5} \ll X^{\frac{60}{43}},\ X^{\frac{615}{43}} \ll (N R)^{29} Q^{10}$;

(6) $X^{\frac{27}{43}} \ll (N R) Q \ll X^{\frac{219}{301}},\ X^{\frac{41}{215}} \ll Q,\ (N R)^{6} Q \ll X^{\frac{135}{43}},\ X^{\frac{205}{473}} \ll (N R),\ (N R)^{-1} Q^{7} \ll X^{\frac{55}{43}},\ X^{\frac{164}{43}} \ll (N R)^{6} Q^{5}$;

(7) $(N R) Q \ll X^{\frac{268}{473}},\ (N R)^{35} Q^{23} \ll X^{\frac{767}{43}},\ X^{\frac{410}{473}} \ll (N R)^{2} Q,\ (N R)^{2} Q^{13} \ll X^{\frac{130}{43}},\ X^{\frac{1476}{43}} \ll (N R)^{70} Q^{59}$;

(8) $(N R) Q \ll X^{\frac{313}{559}},\ (N R)^{41} Q^{27} \ll X^{\frac{902}{43}},\ X^{\frac{492}{559}} \ll (N R)^{2} Q,\ (N R)^{2} Q^{15} \ll X^{\frac{138}{43}},\ X^{\frac{1722}{43}} \ll (N R)^{82} Q^{69}$;

(9) $(N R)^{2} Q \ll X^{\frac{536}{473}},\ (N R)^{70} Q^{11} \ll X^{\frac{1534}{43}},\ X^{\frac{205}{473}} \ll (N R),\ (N R) Q^{6} \ll X^{\frac{65}{43}},\ X^{\frac{738}{43}} \ll (N R)^{35} Q^{12}$.

Assume further that for $(\log X)^{B/ \varepsilon} \leqslant |t| \leqslant 2 X, N(b+i t) R(b+i t) \ll (\log x)^{-B/ \varepsilon}$ and $Q(b+i t) \ll (\log x)^{-B/ \varepsilon}$. Then for $(\log X)^{B/ \varepsilon} \leqslant T \leqslant X$, we have
$$
\left( \min \left(\eta, \frac{1}{T}\right) \right)^{2} \int_{T}^{2 T}|N(b+i t) Q(b+i t) R(b+i t) K(b+i t)|^{2} d t \ll \eta^{2} (\log x)^{-10 B}.
$$
\end{lemma}
\begin{proof}
Let $M(s)=N(s) R(s)$ and $H(s)=Q(s)$ and by Lemma~\ref{l32}, Lemma~\ref{l34} is proved.
\end{proof}

\begin{lemma}\label{l35}
Suppose that $N Q R K=X$, and $N$, $Q$ and $R$ satisfy $X^{\frac{79}{817}} \ll R \ll Q$ and one of the following 9 conditions:

(1) $(Q R) N \ll X^{\frac{673}{1247}},\ X^{\frac{82}{473}} \ll N,\ (Q R)^{29} N^{-1} \ll X^{\frac{427}{43}},\ X^{\frac{246}{817}} \ll (Q R),\ (Q R)^{-1} N^{29} \ll X^{\frac{263}{43}},\ X^{\frac{246}{43}} \ll (Q R)^{12} N^{11}$;

(2) $(Q R) N \ll X^{\frac{223}{387}},\ (Q R)^{29} N^{19} \ll X^{\frac{632}{43}},\ X^{\frac{328}{387}} \ll (Q R)^{2} N,\ (Q R)^{2} N^{11} \ll X^{\frac{120}{43}},\ X^{\frac{1230}{43}} \ll (Q R)^{58} N^{49}$;

(3) $(Q R) N \ll X^{\frac{268}{473}},\ X^{\frac{41}{215}} \ll N,\ (Q R)^{6} N \ll X^{\frac{94}{43}},\ X^{\frac{82}{301}} \ll (Q R) \ll X^{\frac{16}{43}},\ (Q R) N^{8} \ll X^{\frac{98}{43}},\ X^{\frac{123}{43}} \ll (Q R)^{6} N^{5}$;

(4) $(Q R) N \ll X^{\frac{571}{817}},\ X^{\frac{82}{473}} \ll N,\ (Q R)^{12} N \ll X^{\frac{270}{43}},\ X^{\frac{574}{1247}} \ll (Q R),\ (Q R)^{-1} N^{19} \ll X^{\frac{161}{43}},\ X^{\frac{328}{43}} \ll (Q R)^{12} N^{11}$;

(5) $(Q R)^{2} N \ll X^{\frac{446}{387}},\ (Q R)^{58} N^{9} \ll X^{\frac{1264}{43}},\ X^{\frac{164}{387}} \ll (Q R),\ (Q R) N^{5} \ll X^{\frac{60}{43}},\ X^{\frac{615}{43}} \ll (Q R)^{29} N^{10}$;

(6) $X^{\frac{27}{43}} \ll (Q R) N \ll X^{\frac{219}{301}},\ X^{\frac{41}{215}} \ll N,\ (Q R)^{6} N \ll X^{\frac{135}{43}},\ X^{\frac{205}{473}} \ll (Q R),\ (Q R)^{-1} N^{7} \ll X^{\frac{55}{43}},\ X^{\frac{164}{43}} \ll (Q R)^{6} N^{5}$;

(7) $(Q R) N \ll X^{\frac{268}{473}},\ (Q R)^{35} N^{23} \ll X^{\frac{767}{43}},\ X^{\frac{410}{473}} \ll (Q R)^{2} N,\ (Q R)^{2} N^{13} \ll X^{\frac{130}{43}},\ X^{\frac{1476}{43}} \ll (Q R)^{70} N^{59}$;

(8) $(Q R) N \ll X^{\frac{313}{559}},\ (Q R)^{41} N^{27} \ll X^{\frac{902}{43}},\ X^{\frac{492}{559}} \ll (Q R)^{2} N,\ (Q R)^{2} N^{15} \ll X^{\frac{138}{43}},\ X^{\frac{1722}{43}} \ll (Q R)^{82} N^{69}$;

(9) $(Q R)^{2} N \ll X^{\frac{536}{473}},\ (Q R)^{70} N^{11} \ll X^{\frac{1534}{43}},\ X^{\frac{205}{473}} \ll (Q R),\ (Q R) N^{6} \ll X^{\frac{65}{43}},\ X^{\frac{738}{43}} \ll (Q R)^{35} N^{12}$.

Assume further that for $(\log X)^{B/ \varepsilon} \leqslant |t| \leqslant 2 X, Q(b+i t) R(b+i t) \ll (\log x)^{-B/ \varepsilon}$ and $N(b+i t) \ll (\log x)^{-B/ \varepsilon}$. Then for $(\log X)^{B/ \varepsilon} \leqslant T \leqslant X$, we have
$$
\left( \min \left(\eta, \frac{1}{T}\right) \right)^{2} \int_{T}^{2 T}|N(b+i t) Q(b+i t) R(b+i t) K(b+i t)|^{2} d t \ll \eta^{2} (\log x)^{-10 B}.
$$
\end{lemma}
\begin{proof}
Let $M(s)=Q(s) R(s)$ and $H(s)=N(s)$ and by Lemma~\ref{l32}, Lemma~\ref{l35} is proved.
\end{proof}

\begin{lemma}\label{l36}
Suppose that $N Q R K=X$, and $N$, $Q$ and $R$ satisfy $X^{\frac{79}{817}} \ll R \ll Q$ and one of the following 9 conditions:

(1) $N (Q R) \ll X^{\frac{673}{1247}},\ X^{\frac{82}{473}} \ll (Q R),\ N^{29} (Q R)^{-1} \ll X^{\frac{427}{43}},\ X^{\frac{246}{817}} \ll N,\ N^{-1} (Q R)^{29} \ll X^{\frac{263}{43}},\ X^{\frac{246}{43}} \ll N^{12} (Q R)^{11}$;

(2) $N (Q R) \ll X^{\frac{223}{387}},\ N^{29} (Q R)^{19} \ll X^{\frac{632}{43}},\ X^{\frac{328}{387}} \ll N^{2} (Q R),\ N^{2} (Q R)^{11} \ll X^{\frac{120}{43}},\ X^{\frac{1230}{43}} \ll N^{58} (Q R)^{49}$;

(3) $N (Q R) \ll X^{\frac{268}{473}},\ X^{\frac{41}{215}} \ll (Q R),\ N^{6} (Q R) \ll X^{\frac{94}{43}},\ X^{\frac{82}{301}} \ll N \ll X^{\frac{16}{43}},\ N (Q R)^{8} \ll X^{\frac{98}{43}},\ X^{\frac{123}{43}} \ll N^{6} (Q R)^{5}$;

(4) $N (Q R) \ll X^{\frac{571}{817}},\ X^{\frac{82}{473}} \ll (Q R),\ N^{12} (Q R) \ll X^{\frac{270}{43}},\ X^{\frac{574}{1247}} \ll N,\ N^{-1} (Q R)^{19} \ll X^{\frac{161}{43}},\ X^{\frac{328}{43}} \ll N^{12} (Q R)^{11}$;

(5) $N^{2} (Q R) \ll X^{\frac{446}{387}},\ N^{58} (Q R)^{9} \ll X^{\frac{1264}{43}},\ X^{\frac{164}{387}} \ll N,\ N (Q R)^{5} \ll X^{\frac{60}{43}},\ X^{\frac{615}{43}} \ll N^{29} (Q R)^{10}$;

(6) $X^{\frac{27}{43}} \ll N (Q R) \ll X^{\frac{219}{301}},\ X^{\frac{41}{215}} \ll (Q R),\ N^{6} (Q R) \ll X^{\frac{135}{43}},\ X^{\frac{205}{473}} \ll N,\ N^{-1} (Q R)^{7} \ll X^{\frac{55}{43}},\ X^{\frac{164}{43}} \ll N^{6} (Q R)^{5}$;

(7) $N (Q R) \ll X^{\frac{268}{473}},\ N^{35} (Q R)^{23} \ll X^{\frac{767}{43}},\ X^{\frac{410}{473}} \ll N^{2} (Q R),\ N^{2} (Q R)^{13} \ll X^{\frac{130}{43}},\ X^{\frac{1476}{43}} \ll N^{70} (Q R)^{59}$;

(8) $N (Q R) \ll X^{\frac{313}{559}},\ N^{41} (Q R)^{27} \ll X^{\frac{902}{43}},\ X^{\frac{492}{559}} \ll N^{2} (Q R),\ N^{2} (Q R)^{15} \ll X^{\frac{138}{43}},\ X^{\frac{1722}{43}} \ll N^{82} (Q R)^{69}$;

(9) $N^{2} (Q R) \ll X^{\frac{536}{473}},\ N^{70} (Q R)^{11} \ll X^{\frac{1534}{43}},\ X^{\frac{205}{473}} \ll N,\ N (Q R)^{6} \ll X^{\frac{65}{43}},\ X^{\frac{738}{43}} \ll N^{35} (Q R)^{12}$.

Assume further that for $(\log X)^{B/ \varepsilon} \leqslant |t| \leqslant 2 X, Q(b+i t) R(b+i t) \ll (\log x)^{-B/ \varepsilon}$ and $N(b+i t) \ll (\log x)^{-B/ \varepsilon}$. Then for $(\log X)^{B/ \varepsilon} \leqslant T \leqslant X$, we have
$$
\left( \min \left(\eta, \frac{1}{T}\right) \right)^{2} \int_{T}^{2 T}|N(b+i t) Q(b+i t) R(b+i t) K(b+i t)|^{2} d t \ll \eta^{2} (\log x)^{-10 B}.
$$
\end{lemma}
\begin{proof}
Let $M(s)=N(s)$ and $H(s)=Q(s) R(s)$ and by Lemma~\ref{l32}, Lemma~\ref{l36} is proved.
\end{proof}

\section{Arithmetic Information II}
In this section we are looking for more Type-I information. In \cite{Jia120}, Jia used a mean value bound of Deshouillers and Iwaniec \cite{DeshouillersIwaniec}, which makes an approximation to the sixth-power moment of the Riemann zeta-function. Now we shall use another mean value bound of Watt \cite{WattTheorem}, which is stronger than that of Deshouillers and Iwaniec when the length of interval is $n^{\frac{1}{21.5}+\varepsilon}$.

\begin{lemma}\label{l41}
Let $T \geqslant 1$, then
$$ 
\int_{T}^{2 T}\left|L\left(\frac{1}{2}+i t\right)\right|^{4}\left|N\left(\frac{1}{2}+i t\right)\right|^{2} d t \ll \left(T+N^{2} T^{\frac{1}{2}}+N L^{2} T^{-2}\right) T^{\varepsilon^2}.
$$
\end{lemma}
\begin{proof}
For $c_1 L \leqslant T^{\frac{1}{2}}$ and $N \leqslant T$, by the main theorem proved in \cite{WattTheorem} we have
$$ 
\int_{T}^{2 T}\left|L\left(\frac{1}{2}+i t\right)\right|^{4}\left|N\left(\frac{1}{2}+i t\right)\right|^{2} d t \ll \left(T+N^{2} T^{\frac{1}{2}}\right) T^{\varepsilon^2}.
$$
For $c_1 L \leqslant T^{\frac{1}{2}}$ and $N > T$, by the classical mean value estimate (see \cite{Jia120}) we have
\begin{align}
\nonumber \int_{T}^{2 T}\left|L\left(\frac{1}{2}+i t\right)\right|^{4}\left|N\left(\frac{1}{2}+i t\right)\right|^{2} d t \ll&\ \left(L^{2} N + T\right) \left(L N\right)^{\varepsilon^2} \\
\nonumber \ll&\ \left(T+N^{2} T^{\frac{1}{2}}\right) T^{\varepsilon^2}.
\end{align}

When $T^{\frac{1}{2}} < c_1 L \leqslant 2T$, by a reflection principle based on an approximate functional equation, we can get the same bound as above. By this process we may replace $L$ by $L_0 \sim T/L$, so that $L_0 \ll T^{\frac{1}{2}}$. For this, one can see [\cite{DeshouillersIwaniec}, Section 2], [\cite{BHP}, Lemma 2] or [\cite{AlweissLuo}, Lemma 5.2] for a detailed proof with $N \leqslant T$. For $N > T$, by the classical mean value estimate we have
\begin{align}
\nonumber \int_{T}^{2 T}\left|L\left(\frac{1}{2}+i t\right)\right|^{4}\left|N\left(\frac{1}{2}+i t\right)\right|^{2} d t \ll&\ \left(L_0^{2} N + T\right) \left(L_0 N\right)^{\varepsilon^2} + \text{Error} \\
\nonumber \ll&\ \left(T+N^{2} T^{\frac{1}{2}}\right) T^{\varepsilon^2}.
\end{align}

When $2T < c_1 L$, by the method in \cite{Jia120} we have
$$ 
\int_{T}^{2 T}\left|L\left(\frac{1}{2}+i t\right)\right|^{4}\left|N\left(\frac{1}{2}+i t\right)\right|^{2} d t \ll N L^{2} T^{-2}.
$$
(Actually we have a better bound $L^2 T^{-4} (N+T)$ in this case, but this won't bring any improvement to the next lemma.)

Finally, by combining all the cases above, Lemma~\ref{l41} is proved.
\end{proof}

\begin{lemma}\label{l42}
Suppose that $M H L=X$, and $M$ and $H$ satisfy one of the following 2 conditions:

(1) $M^{2} H \ll X^{2-\frac{20.5}{21.5}},\ M^{4} H^{6} \ll X^{4-\frac{20.5}{21.5}},\ H^{4} \ll X^{4-\frac{61.5}{21.5}}$;

(2) $M \ll X^{\frac{45}{86}},\ H \ll X^{\frac{41}{344}} $.

Then for $\sqrt{L} \leqslant T \leqslant X$, we have
$$
\left( \min \left(\eta, \frac{1}{T}\right) \right)^{2} \int_{T}^{2 T}|M(b+i t) H(b+i t) L(b+i t)|^{2} d t \ll \eta^{2} x^{-2 \varepsilon^2}.
$$
\end{lemma}
\begin{proof}
We can prove this by using Lemma~\ref{l41} and the methods in [\cite{Jia120}, Lemmas 3,4]. For condition (1) we apply Lemma~\ref{l41} to $L$ and $N$, while for condition (2) we apply Lemma~\ref{l41} to $L$ and $N^2$. One can see \cite{Jia120} for an explanation.
\end{proof}

\begin{lemma}\label{l43}
Suppose that $N Q R L=X$, and $N$, $Q$ and $R$ satisfy $X^{\frac{79}{817}} \ll R \ll Q$ and one of the following 9 conditions:

(1) $(N Q)^{2} R \ll X^{2-\frac{20.5}{21.5}},\ (N Q)^{4} R^{6} \ll X^{4-\frac{20.5}{21.5}},\ R^{4} \ll X^{4-\frac{61.5}{21.5}} $;

(2) $(N Q) \ll X^{\frac{45}{86}},\ R \ll X^{\frac{41}{344}} $;

(3) $X^{\frac{55}{129}} \ll (N Q) \ll X^{\frac{20}{43}},\ R \ll (N Q)^{-\frac{1}{5}} X^{\frac{12}{43}} $;

(4) $X^{\frac{20}{43}} \ll (N Q) \ll X^{\frac{41}{86}},\ R \ll (N Q)^{\frac{1}{7}} X^{\frac{55}{301}} $;

(5) $X^{\frac{41}{86}} \ll (N Q) \ll X^{\frac{227}{473}},\ R \ll (N Q)^{-1} X^{\frac{219}{301}} $;

(6) $X^{\frac{227}{473}} \ll (N Q) \ll X^{\frac{2333}{4859}},\ R \ll (N Q)^{-\frac{58}{9}} X^{\frac{1264}{387}} $;

(7) $X^{\frac{2333}{4859}} \ll (N Q) \ll X^{\frac{2501}{5203}},\ R \ll (N Q)^{-\frac{1}{6}} X^{\frac{65}{258}} $;

(8) $X^{\frac{2501}{5203}} \ll (N Q) \ll X^{\frac{499}{1032}},\ R \ll (N Q)^{-2} X^{\frac{536}{473}} $;

(9) $X^{\frac{499}{1032}} \ll (N Q) \ll X^{\frac{28277}{57190}},\ R \ll (N Q)^{-\frac{70}{11}} X^{\frac{1534}{473}} $.

Assume further that for $\sqrt{L} \leqslant |t| \leqslant 2 X, N(b+i t) Q(b+i t) \ll (\log x)^{-B/ \varepsilon}$ and $R(b+i t) \ll (\log x)^{-B/ \varepsilon}$. Then for $\sqrt{L} \leqslant T \leqslant X$, we have
$$
\left( \min \left(\eta, \frac{1}{T}\right) \right)^{2} \int_{T}^{2 T}|N(b+i t) Q(b+i t) R(b+i t) L(b+i t)|^{2} d t \ll \eta^{2} (\log x)^{-10 B}.
$$
\end{lemma}
\begin{proof}
The proof is similar to that of [\cite{Jia120}, Lemmas 14,15]. Let $M(s)=N(s)Q(s)$ and $H(s)=R(s)$. Then by Lemma~\ref{l42} and Lemma~\ref{l32}, Lemma~\ref{l43} is proved.
\end{proof}

\begin{lemma}\label{l44}
Suppose that $N Q R L=X$, and $N$, $Q$ and $R$ satisfy $X^{\frac{79}{817}} \ll R \ll Q$ and one of the following 9 conditions:

(1) $(Q R)^{2} N \ll X^{2-\frac{20.5}{21.5}},\ (Q R)^{4} N^{6} \ll X^{4-\frac{20.5}{21.5}},\ N^{4} \ll X^{4-\frac{61.5}{21.5}} $;

(2) $(Q R) \ll X^{\frac{45}{86}},\ N \ll X^{\frac{41}{344}} $;

(3) $X^{\frac{55}{129}} \ll (Q R) \ll X^{\frac{20}{43}},\ N \ll (Q R)^{-\frac{1}{5}} X^{\frac{12}{43}} $;

(4) $X^{\frac{20}{43}} \ll (Q R) \ll X^{\frac{41}{86}},\ N \ll (Q R)^{\frac{1}{7}} X^{\frac{55}{301}} $;

(5) $X^{\frac{41}{86}} \ll (Q R) \ll X^{\frac{227}{473}},\ N \ll (Q R)^{-1} X^{\frac{219}{301}} $;

(6) $X^{\frac{227}{473}} \ll (Q R) \ll X^{\frac{2333}{4859}},\ N \ll (Q R)^{-\frac{58}{9}} X^{\frac{1264}{387}} $;

(7) $X^{\frac{2333}{4859}} \ll (Q R) \ll X^{\frac{2501}{5203}},\ N \ll (Q R)^{-\frac{1}{6}} X^{\frac{65}{258}} $;

(8) $X^{\frac{2501}{5203}} \ll (Q R) \ll X^{\frac{499}{1032}},\ N \ll (Q R)^{-2} X^{\frac{536}{473}} $;

(9) $X^{\frac{499}{1032}} \ll (Q R) \ll X^{\frac{28277}{57190}},\ N \ll (Q R)^{-\frac{70}{11}} X^{\frac{1534}{473}} $.

Assume further that for $\sqrt{L} \leqslant |t| \leqslant 2 X, Q(b+i t) R(b+i t) \ll (\log x)^{-B/ \varepsilon}$ and $N(b+i t) \ll (\log x)^{-B/ \varepsilon}$. Then for $\sqrt{L} \leqslant T \leqslant X$, we have
$$
\left( \min \left(\eta, \frac{1}{T}\right) \right)^{2} \int_{T}^{2 T}|N(b+i t) Q(b+i t) R(b+i t) L(b+i t)|^{2} d t \ll \eta^{2} (\log x)^{-10 B}.
$$
\end{lemma}
\begin{proof}
Let $M(s)=Q(s)R(s)$ and $H(s)=N(s)$. Then by Lemma~\ref{l42} and Lemma~\ref{l32}, Lemma~\ref{l44} is proved.
\end{proof}

\section{Asymptotic formulas}

\begin{lemma}\label{l51}
Suppose that $M \ll X^{\frac{45}{86}}$ and $H \ll X^{\frac{41}{344}}$. Then for real numbers $x \in [X, 2 X]$, except for $O\left(X (\log X)^{-B} \right)$ values, we have
$$
\sum_{\substack{m \sim M \\ h \sim H}} a_m b_h \left(\sum_{x<m h l \leqslant x+\eta x} 1-\frac{\eta x}{m h}\right)\ll \eta x (\log x)^{-B}.
$$
\end{lemma}
\begin{proof}
The proof is similar to that of [\cite{Jia120}, Lemma 16] where the condition (2) of Lemma~\ref{l42} is used.
\end{proof}

\begin{lemma}\label{l52}
Suppose that $N Q R L = X$, and $N$, $Q$, $R$ satisfy one of the 18 conditions in Lemma~\ref{l43} and Lemma~\ref{l44}. Assume further that for $\sqrt{L} \leqslant |t| \leqslant 2 X, N(b+i t) Q(b+i t) R(b+i t) \ll (\log x)^{-B/ \varepsilon}$. Then for real numbers $x \in [X, 2 X]$, except for $O\left(X (\log X)^{-B} \right)$ values, we have
$$
\sum_{\substack{n \sim N \\ q \sim Q \\ r \sim R}} a_n b_q c_r \left(\sum_{x< n q r l \leqslant x+\eta x} 1-\frac{\eta x}{n q r}\right)\ll \eta x (\log x)^{-B}.
$$
\end{lemma}
\begin{proof}
The proof is similar to that of [\cite{Jia120}, Lemma 17].
\end{proof}

\begin{lemma}\label{l61}
Suppose that $X^{\frac{738}{817}} \ll M \ll X^{1-\delta}$, $a_m \geqslant 0$, and $a_m = 0$ if $m$ has a prime factor $<X^{\delta}$. Then for real numbers $x \in [X, 2 X]$, except for $O\left(X (\log X)^{-B} \right)$ values, we have
$$
\sum_{\substack{x<m p \leqslant x+\eta x \\
m \sim M}} a_m =\eta\left(1+O\left(\frac{1}{\log x}\right)\right) \sum_{\substack{x<m p \leqslant 2 x \\ m \sim M}} a_m +O\left(\eta x (\log x)^{-B}\right).
$$
\end{lemma}
\begin{proof}
The proof is similar to that of [\cite{Jia120}, Lemma 18] where Lemma~\ref{l31} is used.
\end{proof}

\begin{lemma}\label{l62} Suppose that $M H K=X$, $c_k = \Lambda(k)$, $X^{\frac{79}{817}} \ll M H \ll X^{\frac{738}{817}}$, and $a_m, b_h \geqslant 0$. Suppose that $a_m=0$ if $m$ has a prime factor $<X^{\delta}$, and $b_h=0$ if $h$ has a prime factor $<X^{\delta}$. If we have
$$
\left( \min \left(\eta, \frac{1}{T}\right) \right)^{2} \int_{T}^{2 T}|M(b+i t) H(b+i t) K(b+i t)|^{2} d t \ll \eta^{2} (\log x)^{-10 B}
$$
for $(\log X)^{B/ \varepsilon} \leqslant T \leqslant X$, then for real numbers $x \in [X, 2 X]$, except for $O\left(X (\log X)^{-B} \right)$ values, we have
$$
\sum_{\substack{x<m h p \leqslant x+\eta x \\
m \sim M \\ h \sim H}} a_m b_h =\eta\left(1+O\left(\frac{1}{\log x}\right)\right) \sum_{\substack{x<m h p \leqslant 2 x \\ m \sim M \\ h \sim H}} a_m b_h +O\left(\eta x (\log x)^{-B}\right).
$$
\end{lemma}
\begin{proof}
The proof is similar to that of [\cite{Jia1132}, Lemma 4].
\end{proof}

\begin{lemma}\label{l63} Suppose that $M N H K=X$, $d_k = \Lambda(k)$, and $a_m, b_n, c_h \geqslant 0$. Suppose that $a_m=0$ if $m$ has a prime factor $<X^{\delta}$, $b_n=0$ if $n$ has a prime factor $<X^{\delta}$, and $c_h=0$ if $h$ has a prime factor $<X^{\delta}$. If we have
$$
\left( \min \left(\eta, \frac{1}{T}\right) \right)^{2} \int_{T}^{2 T}|M(b+i t) N(b+i t) H(b+i t) K(b+i t)|^{2} d t \ll \eta^{2} (\log x)^{-10 B}
$$
for $(\log X)^{B/ \varepsilon} \leqslant T \leqslant X$, then for real numbers $x \in [X, 2 X]$, except for $O\left(X (\log X)^{-B} \right)$ values, we have
$$
\sum_{\substack{x<m n h p \leqslant x+\eta x \\
m \sim M \\ n \sim N \\ h \sim H}} a_m b_n c_h =\eta\left(1+O\left(\frac{1}{\log x}\right)\right) \sum_{\substack{x<m h p \leqslant 2 x \\ m \sim M \\ n \sim N \\ h \sim H}} a_m b_n c_h +O\left(\eta x (\log x)^{-B}\right).
$$
\end{lemma}
\begin{proof}
The proof is similar to that of [\cite{Jia1132}, Lemma 4].
\end{proof}

\begin{lemma}\label{TypeI}
Suppose that $P_0 P_1 P_2 \cdots P_n = X$, $P_n < \cdots < P_2 < P_1 < X^{\frac{1}{2}}$ and $p_i \sim P_i$ for $1 \leqslant i$. Suppose that $P_1, P_2, \ldots, P_n, X^{10^{-1000}}$ can be partitioned into $M, H$ satisfy Lemma~\ref{l51} or $N, Q, R$ satisfy Lemma~\ref{l52}, Then for real numbers $x \in [X, 2 X]$, except for $O\left(X (\log X)^{-B} \right)$ values, we have
$$
\sum_{(t_1, \ldots, t_n)} S\left(\mathcal{A}_{p_1 p_2 \cdots p_n}, X^{\frac{79}{817}}\right) = \eta \sum_{(t_1, \ldots, t_n)} S\left(\mathcal{B}_{p_1 p_2 \cdots p_n}, X^{\frac{79}{817}}\right) + O\left(\frac{\varepsilon \eta x}{\log x}\right).
$$
\end{lemma}
\begin{proof}
By Buchstab's identity, we have
$$
\nonumber \sum_{(t_1, \ldots, t_n)} S\left(\mathcal{A}_{p_1 p_2 \cdots p_n}, X^{\frac{79}{817}}\right) = \sum_{(t_1, \ldots, t_n)} S\left(\mathcal{A}_{p_1 p_2 \cdots p_n}, X^{\delta}\right) - \sum_{\substack{(t_1, \ldots, t_n) \\ \delta \leqslant t_{n+1} < \frac{79}{817}}} S\left(\mathcal{A}_{p_1 p_2 \cdots p_n p_{n+1}}, p_{n+1}\right).
$$
We use Iwaniec's linear sieve (taking $D = X^{10^{-1000}}, z = X^{\delta}$) together with Lemma~\ref{l51} or Lemma~\ref{l52} (depending on the region that $P_1, P_2, \ldots, P_n, X^{10^{-1000}}$ are partitioned into) to deal with the first term on the right-hand side. The functions $F(s)$ and $f(s)$ defined by the following differential-difference equation
\begin{align*}
\begin{cases}
F(s)=\frac{2 e^{\gamma}}{s}, \quad f(s)=0, \quad &0<s \leqslant 2,\\
(s F(s))^{\prime}=f(s-1), \quad(s f(s))^{\prime}=F(s-1), \quad &s \geqslant 2 ,
\end{cases}
\end{align*}
give the same value $1$ with negligible error $O\left(\varepsilon^2\right)$, where $\gamma$ is the Euler's constant. Hence we have
$$
\sum_{(t_1, \ldots, t_n)} S\left(\mathcal{A}_{p_1 p_2 \cdots p_n}, X^{\delta}\right) = \eta \sum_{(t_1, \ldots, t_n)} S\left(\mathcal{B}_{p_1 p_2 \cdots p_n}, X^{\delta}\right) + O\left(\frac{\varepsilon \eta x}{\log x}\right).
$$
Let $M = \frac{X}{P_{n+1}}$. Clearly we have $X^{\frac{738}{817}} \ll M \ll X^{1-\delta}$. By Lemma~\ref{l61}, we have
$$
\sum_{\substack{(t_1, \ldots, t_n) \\ \delta \leqslant t_{n+1} < \frac{79}{817}}} S\left(\mathcal{A}_{p_1 p_2 \cdots p_n p_{n+1}}, p_{n+1}\right) = \eta \sum_{\substack{(t_1, \ldots, t_n) \\ \delta \leqslant t_{n+1} < \frac{79}{817}}} S\left(\mathcal{B}_{p_1 p_2 \cdots p_n p_{n+1}}, p_{n+1}\right) + O\left(\frac{\varepsilon \eta x}{\log x}\right).
$$
Now, Lemma~\ref{TypeI} is proved by another application of Buchstab's identity
$$
\sum_{(t_1, \ldots, t_n)} S\left(\mathcal{B}_{p_1 p_2 \cdots p_n}, X^{\delta}\right) - \sum_{\substack{(t_1, \ldots, t_n) \\ \delta \leqslant t_{n+1} < \frac{79}{817}}} S\left(\mathcal{B}_{p_1 p_2 \cdots p_n p_{n+1}}, p_{n+1}\right) = \nonumber \sum_{(t_1, \ldots, t_n)} S\left(\mathcal{B}_{p_1 p_2 \cdots p_n}, X^{\frac{79}{817}}\right).
$$
We remark that in many places, we ignore the additional variable $X^{10^{-1000}}$ when applying this lemma since it only produces an extremely small difference (less than $10^{-10}$).
\end{proof}

\section{The final decomposition}
In this section, sets $\mathcal{A}$ and $\mathcal{B}$ are defined respectively. Let $\omega(u)$ denote the Buchstab function determined by the following differential-difference equation
\begin{align*}
\begin{cases}
\omega(u)=\frac{1}{u}, & \quad 1 \leqslant u \leqslant 2, \\
(u \omega(u))^{\prime}= \omega(u-1), & \quad u \geqslant 2 .
\end{cases}
\end{align*}
Moreover, we have the upper and lower bounds for $\omega(u)$:
\begin{align*}
\omega(u) \geqslant \omega_{0}(u) =
\begin{cases}
\frac{1}{u}, & \quad 1 \leqslant u < 2, \\
\frac{1+\log(u-1)}{u}, & \quad 2 \leqslant u < 3, \\
\frac{1+\log(u-1)}{u} + \frac{1}{u} \int_{2}^{u-1}\frac{\log(t-1)}{t} d t, & \quad 3 \leqslant u < 4, \\
0.5612, & \quad u \geqslant 4, \\
\end{cases}
\end{align*}
\begin{align*}
\omega(u) \leqslant \omega_{1}(u) =
\begin{cases}
\frac{1}{u}, & \quad 1 \leqslant u < 2, \\
\frac{1+\log(u-1)}{u}, & \quad 2 \leqslant u < 3, \\
\frac{1+\log(u-1)}{u} + \frac{1}{u} \int_{2}^{u-1}\frac{\log(t-1)}{t} d t, & \quad 3 \leqslant u < 4, \\
0.5617, & \quad u \geqslant 4. \\
\end{cases}
\end{align*}
We shall use $\omega_0(u)$ and $\omega_1(u)$ to give numerical bounds for some sieve functions discussed below.

By [\cite{Jia120}, Lemma 21], we know that, for sufficiently large $z$,
\begin{equation}
S\left(\mathcal{B}, z\right) = \sum_{\substack{a \in \mathcal{B} \\ (a, P(z))=1}} 1 = (1+o(1)) \frac{x}{\log z} \omega\left(\frac{\log x}{\log z}\right),
\end{equation}
and we expect that the similar relation also holds for $S\left(\mathcal{A}, z\right)$:
\begin{equation}
S\left(\mathcal{A}, z\right) = \sum_{\substack{a \in \mathcal{A} \\ (a, P(z))=1}} 1 = (1+o(1)) \frac{\eta x}{\log z} \omega\left(\frac{\log x}{\log z}\right).
\end{equation}
If (3) holds for $S\left(\mathcal{A}, z\right)$, then we can deduce (5) easily from (3) and (4). Otherwise we must drop this $S\left(\mathcal{A}, z\right)$. We define the \textit{loss} from this term by the size of corresponding $S\left(\mathcal{B}, z\right)$:
\begin{equation}
S\left(\mathcal{B}, z\right) = (\textit{loss}+o(1)) \frac{\eta x}{\log x}.
\end{equation}
We note that we can only drop positive parts and the total loss of the dropped parts must be less than $1$.

Beginning with (2), we can easily give asymptotic formulas for $S_1$ and $S_2$ by Lemma~\ref{TypeI}. Before estimating $S_{3}$, we first define the disjoint regions $U_{01}$--$U_{03}$ as
\begin{align}
\nonumber U_{01} (t_1, t_2) :=&\ \left\{(t_1, t_2) \in U_{0101} \cup U_{0102} \cup \cdots \cup U_{0109}, \right. \\
\nonumber & \left. \quad \frac{79}{817} \leqslant t_1 < \frac{1}{2},\ \frac{79}{817} \leqslant t_2 < \min\left(t_1, \frac{1}{2}(1-t_1) \right) \right\}, \\
\nonumber U_{02} (t_1, t_2) :=&\ \left\{(t_1, t_2) \in U_{0201} \cup U_{0202} \cup \cdots \cup U_{0237},\ (t_1, t_2) \notin U_{01}, \right. \\
\nonumber & \left. \quad \frac{79}{817} \leqslant t_1 < \frac{1}{2},\ \frac{79}{817} \leqslant t_2 < \min\left(t_1, \frac{1}{2}(1-t_1) \right) \right\}, \\
\nonumber U_{03} (t_1, t_2) :=&\ \left\{(t_1, t_2) \in U_{0301} \cup U_{0302} \cup \cdots \cup U_{0318},\ (t_1, t_2) \notin U_{01} \cup U_{02}, \right. \\
\nonumber & \left. \quad \frac{79}{817} \leqslant t_1 < \frac{1}{2},\ \frac{79}{817} \leqslant t_2 < \min\left(t_1, \frac{1}{2}(1-t_1) \right) \right\},
\end{align}
where
\begin{align}
\nonumber U_{0101} (t_1, t_2) :=&\ \left\{t_1 + t_2 \leqslant \frac{673}{1247},\ \frac{82}{473} \leqslant t_2,\ 29 t_1 - t_2 \leqslant \frac{427}{43},\ \frac{246}{817} \leqslant t_1,\ 29 t_2 - t_1 \leqslant \frac{263}{43},\ \frac{246}{43} \leqslant 12 t_1 + 11 t_2 \right\}, \\
\nonumber U_{0102}(t_1, t_2) :=&\ \left\{ t_1 + t_2 \leqslant \frac{223}{387},\ 29 t_1 + 19 t_2 \leqslant \frac{632}{43},\ \frac{328}{387} \leqslant 2 t_1 + t_2,\ 2 t_1 + 11 t_2 \leqslant \frac{120}{43},\ \frac{1230}{43} \leqslant 58 t_1 + 49 t_2 \right\}, \\
\nonumber U_{0103}(t_1, t_2) :=&\ \left\{ t_1 + t_2 \leqslant \frac{268}{473},\ \frac{41}{215} \leqslant t_2,\ 6 t_1 + t_2 \leqslant \frac{94}{43},\ \frac{82}{301} \leqslant t_1 \leqslant \frac{16}{43},\ t_1 + 8 t_2 \leqslant \frac{98}{43},\ \frac{123}{43} \leqslant 6 t_1 + 5 t_2 \right\}, \\
\nonumber U_{0104}(t_1, t_2) :=&\ \left\{ t_1 + t_2 \leqslant \frac{571}{817},\ \frac{82}{473} \leqslant t_2,\ 12 t_1 + t_2 \leqslant \frac{270}{43},\ \frac{574}{1247} \leqslant t_1,\ 19 t_2 - t_1 \leqslant \frac{161}{43},\ \frac{328}{43} \leqslant 12 t_1 + 11 t_2 \right\}, \\
\nonumber U_{0105}(t_1, t_2) :=&\ \left\{ 2 t_1 + t_2 \leqslant \frac{446}{387},\ 58 t_1 + 9 t_2 \leqslant \frac{1264}{43},\ \frac{164}{387} \leqslant t_1,\ t_1 + 5 t_2 \leqslant \frac{60}{43},\ \frac{615}{43} \leqslant 29 t_1 + 10 t_2 \right\}, \\
\nonumber U_{0106}(t_1, t_2) :=&\ \left\{ \frac{27}{43} \leqslant t_1 + t_2 \leqslant \frac{219}{301},\ \frac{41}{215} \leqslant t_2,\ 6 t_1 + t_2 \leqslant \frac{135}{43},\ \frac{205}{473} \leqslant t_1,\ 7 t_2 - t_1 \leqslant \frac{55}{43},\ \frac{164}{43} \leqslant 6 t_1 + 5 t_2 \right\}, \\
\nonumber U_{0107}(t_1, t_2) :=&\ \left\{ t_1 + t_2 \leqslant \frac{268}{473},\ 35 t_1 + 23 t_2 \leqslant \frac{767}{43},\ \frac{410}{473} \leqslant 2 t_1 + t_2,\ 2 t_1 + 13 t_2 \leqslant \frac{130}{43},\ \frac{1476}{43} \leqslant 70 t_1 + 59 t_2 \right\}, \\
\nonumber U_{0108}(t_1, t_2) :=&\ \left\{ t_1 + t_2 \leqslant \frac{313}{559},\ 41 t_1 + 27 t_2 \leqslant \frac{902}{43},\ \frac{492}{559} \leqslant 2 t_1 + t_2,\ 2 t_1 + 15 t_2 \leqslant \frac{138}{43},\ \frac{1722}{43} \leqslant 82 t_1 + 69 t_2 \right\}, \\
\nonumber U_{0109}(t_1, t_2) :=&\ \left\{ 2 t_1 + t_2 \leqslant \frac{536}{473},\ 70 t_1 + 11 t_2 \leqslant \frac{1534}{43},\ \frac{205}{473} \leqslant t_1,\ t_1 + 6 t_2 \leqslant \frac{65}{43},\ \frac{738}{43} \leqslant 35 t_1 + 12 t_2 \right\}, \\
\nonumber U_{0201}(t_1, t_2) :=&\ \left\{ \frac{27267}{66994} \leqslant t_1 + t_2 \leqslant \frac{231}{559},\ \frac{79}{817} \leqslant t_2 \leqslant -\frac{35}{23} (t_1 + t_2) + \frac{767}{989} \right\}, \\
\nonumber U_{0202}(t_1, t_2) :=&\ \left\{ \frac{231}{559} \leqslant t_1 + t_2 \leqslant \frac{3275}{7826},\ \frac{79}{817} \leqslant t_2 \leqslant - (t_1 + t_2) + \frac{313}{559} \right\}, \\
\nonumber U_{0203}(t_1, t_2) :=&\ \left\{ \frac{3275}{7826} \leqslant t_1 + t_2 \leqslant \frac{15005}{33497},\ \frac{79}{817} \leqslant t_2 \leqslant -\frac{41}{27} (t_1 + t_2) + \frac{902}{1161} \right\}, \\
\nonumber U_{0204}(t_1, t_2) :=&\ \left\{ \frac{13074}{28595} \leqslant t_1 + t_2 \leqslant \frac{20}{43},\ \frac{79}{817} \leqslant t_2 \leqslant -\frac{1}{5} (t_1 + t_2) + \frac{12}{43} \right\}, \\
\nonumber U_{0205}(t_1, t_2) :=&\ \left\{ \frac{20}{43} \leqslant t_1 + t_2 \leqslant \frac{41}{86},\ \frac{79}{817} \leqslant t_2 \leqslant \frac{1}{7} (t_1 + t_2) + \frac{55}{301} \right\}, \\
\nonumber U_{0206}(t_1, t_2) :=&\ \left\{ \frac{41}{86} \leqslant t_1 + t_2 \leqslant \frac{227}{473},\ \frac{79}{817} \leqslant t_2 \leqslant - (t_1 + t_2) + \frac{219}{301} \right\}, \\
\nonumber U_{0207}(t_1, t_2) :=&\ \left\{ \frac{227}{473} \leqslant t_1 + t_2 \leqslant \frac{2333}{4859},\ \frac{79}{817} \leqslant t_2 \leqslant -\frac{58}{9} (t_1 + t_2) + \frac{1264}{387} \right\}, \\
\nonumber U_{0208}(t_1, t_2) :=&\ \left\{ \frac{2333}{4859} \leqslant t_1 + t_2 \leqslant \frac{2501}{5203},\ \frac{79}{817} \leqslant t_2 \leqslant -\frac{1}{6} (t_1 + t_2) + \frac{65}{258} \right\}, \\
\nonumber U_{0209}(t_1, t_2) :=&\ \left\{ \frac{2501}{5203} \leqslant t_1 + t_2 \leqslant \frac{499}{1032},\ \frac{79}{817} \leqslant t_2 \leqslant - 2 (t_1 + t_2) + \frac{536}{473} \right\}, \\
\nonumber U_{0210}(t_1, t_2) :=&\ \left\{ \frac{499}{1032} \leqslant t_1 + t_2 \leqslant \frac{28277}{57190},\ \frac{79}{817} \leqslant t_2 \leqslant -\frac{70}{11} (t_1 + t_2) + \frac{1534}{473} \right\}, \\
\nonumber U_{0211}(t_1, t_2) :=&\ \left\{ (t_1 + t_2) + t_2 \leqslant \frac{673}{1247},\ \frac{82}{473} \leqslant t_2,\ 29 (t_1 + t_2) - t_2 \leqslant \frac{427}{43},\ \frac{246}{817} \leqslant \left(t_1 + \frac{79}{817}\right), \right. \\
\nonumber & \left. \quad 29 t_2 - \left(t_1 + \frac{79}{817}\right) \leqslant \frac{263}{43},\ \frac{246}{43} \leqslant 12 \left(t_1 + \frac{79}{817}\right) + 11 t_2 \right\}, \\
\nonumber U_{0212}(t_1, t_2) :=&\ \left\{ (t_1 + t_2) + t_2 \leqslant \frac{223}{387},\ 29 (t_1 + t_2) + 19 t_2 \leqslant \frac{632}{43},\ \frac{328}{387} \leqslant 2 \left(t_1 + \frac{79}{817}\right) + t_2, \right. \\
\nonumber & \left. \quad 2 (t_1 + t_2) + 11 t_2 \leqslant \frac{120}{43},\ \frac{1230}{43} \leqslant 58 \left(t_1 + \frac{79}{817}\right) + 49 t_2 \right\}, \\
\nonumber U_{0213}(t_1, t_2) :=&\ \left\{ (t_1 + t_2) + t_2 \leqslant \frac{268}{473},\ \frac{41}{215} \leqslant t_2,\ 6 (t_1 + t_2) + t_2 \leqslant \frac{94}{43},\ \frac{82}{301} \leqslant \left(t_1 + \frac{79}{817}\right), \right. \\
\nonumber & \left. \quad (t_1 + t_2) \leqslant \frac{16}{43},\ (t_1 + t_2) + 8 t_2 \leqslant \frac{98}{43},\ \frac{123}{43} \leqslant 6 \left(t_1 + \frac{79}{817}\right) + 5 t_2 \right\}, \\
\nonumber U_{0214}(t_1, t_2) :=&\ \left\{ (t_1 + t_2) + t_2 \leqslant \frac{571}{817},\ \frac{82}{473} \leqslant t_2,\ 12 (t_1 + t_2) + t_2 \leqslant \frac{270}{43},\ \frac{574}{1247} \leqslant \left(t_1 + \frac{79}{817}\right), \right. \\
\nonumber & \left. \quad 19 t_2 - \left(t_1 + \frac{79}{817}\right) \leqslant \frac{161}{43},\ \frac{328}{43} \leqslant 12 \left(t_1 + \frac{79}{817}\right) + 11 t_2 \right\}, \\
\nonumber U_{0215}(t_1, t_2) :=&\ \left\{ 2 (t_1 + t_2) + t_2 \leqslant \frac{446}{387},\ 58 (t_1 + t_2) + 9 t_2 \leqslant \frac{1264}{43},\ \frac{164}{387} \leqslant \left(t_1 + \frac{79}{817}\right), \right. \\
\nonumber & \left. \quad (t_1 + t_2) + 5 t_2 \leqslant \frac{60}{43},\ \frac{615}{43} \leqslant 29 \left(t_1 + \frac{79}{817}\right) + 10 t_2 \right\}, \\
\nonumber U_{0216}(t_1, t_2) :=&\ \left\{ \frac{27}{43} \leqslant \left(t_1 + \frac{79}{817}\right) +t_2,\ (t_1 + t_2) + t_2 \leqslant \frac{219}{301},\ \frac{41}{215} \leqslant t_2,\ 6 (t_1 + t_2) + t_2 \leqslant \frac{135}{43}, \right. \\
\nonumber & \left. \quad \frac{205}{473} \leqslant \left(t_1 + \frac{79}{817}\right),\ 7 t_2 - \left(t_1 + \frac{79}{817}\right) \leqslant \frac{55}{43},\ \frac{164}{43} \leqslant 6 \left(t_1 + \frac{79}{817}\right) + 5 t_2 \right\}, \\
\nonumber U_{0217}(t_1, t_2) :=&\ \left\{ (t_1 + t_2) + t_2 \leqslant \frac{268}{473},\ 35 (t_1 + t_2) + 23 t_2 \leqslant \frac{767}{43},\ \frac{410}{473} \leqslant 2 \left(t_1 + \frac{79}{817}\right) + t_2, \right. \\
\nonumber & \left. \quad 2 (t_1 + t_2) + 13 t_2 \leqslant \frac{130}{43},\ \frac{1476}{43} \leqslant 70 \left(t_1 + \frac{79}{817}\right) + 59 t_2 \right\}, \\
\nonumber U_{0218}(t_1, t_2) :=&\ \left\{ (t_1 + t_2) + t_2 \leqslant \frac{313}{559},\ 41 (t_1 + t_2) + 27 t_2 \leqslant \frac{902}{43},\ \frac{492}{559} \leqslant 2 \left(t_1 + \frac{79}{817}\right) + t_2, \right. \\
\nonumber & \left. \quad 2 (t_1 + t_2) + 15 t_2 \leqslant \frac{138}{43},\ \frac{1722}{43} \leqslant 82 \left(t_1 + \frac{79}{817}\right) + 69 t_2 \right\}, \\
\nonumber U_{0219}(t_1, t_2) :=&\ \left\{ 2 (t_1 + t_2) + t_2 \leqslant \frac{536}{473},\ 70 (t_1 + t_2) + 11 t_2 \leqslant \frac{1534}{43},\ \frac{205}{473} \leqslant \left(t_1 + \frac{79}{817}\right), \right. \\
\nonumber & \left. \quad (t_1 + t_2) + 6 t_2 \leqslant \frac{65}{43},\ \frac{738}{43} \leqslant 35 \left(t_1 + \frac{79}{817}\right) + 12 t_2 \right\}, \\
\nonumber U_{0220}(t_1, t_2) :=&\ \left\{ (t_2 + t_2) + t_1 \leqslant \frac{673}{1247},\ \frac{82}{473} \leqslant t_1,\ 29 (t_2 + t_2) - t_1 \leqslant \frac{427}{43},\ \frac{246}{817} \leqslant \left(t_2 + \frac{79}{817}\right), \right. \\
\nonumber & \left. \quad 29 t_1 - \left(t_2 + \frac{79}{817}\right) \leqslant \frac{263}{43},\ \frac{246}{43} \leqslant 12 \left(t_2 + \frac{79}{817}\right) + 11 t_1 \right\}, \\
\nonumber U_{0221}(t_1, t_2) :=&\ \left\{ (t_2 + t_2) + t_1 \leqslant \frac{223}{387},\ 29 (t_2 + t_2) + 19 t_1 \leqslant \frac{632}{43},\ \frac{328}{387} \leqslant 2 \left(t_2 + \frac{79}{817}\right) + t_1, \right. \\
\nonumber & \left. \quad 2 (t_2 + t_2) + 11 t_1 \leqslant \frac{120}{43},\ \frac{1230}{43} \leqslant 58 \left(t_2 + \frac{79}{817}\right) + 49 t_1 \right\}, \\
\nonumber U_{0222}(t_1, t_2) :=&\ \left\{ (t_2 + t_2) + t_1 \leqslant \frac{268}{473},\ \frac{41}{215} \leqslant t_1,\ 6 (t_2 + t_2) + t_1 \leqslant \frac{94}{43},\ \frac{82}{301} \leqslant \left(t_2 + \frac{79}{817}\right), \right. \\
\nonumber & \left. \quad (t_2 + t_2) \leqslant \frac{16}{43},\ (t_2 + t_2) + 8 t_1 \leqslant \frac{98}{43},\ \frac{123}{43} \leqslant 6 \left(t_2 + \frac{79}{817}\right) + 5 t_1 \right\}, \\
\nonumber U_{0223}(t_1, t_2) :=&\ \left\{ (t_2 + t_2) + t_1 \leqslant \frac{571}{817},\ \frac{82}{473} \leqslant t_1,\ 12 (t_2 + t_2) + t_1 \leqslant \frac{270}{43},\ \frac{574}{1247} \leqslant \left(t_2 + \frac{79}{817}\right), \right. \\
\nonumber & \left. \quad 19 t_1 - \left(t_2 + \frac{79}{817}\right) \leqslant \frac{161}{43},\ \frac{328}{43} \leqslant 12 \left(t_2 + \frac{79}{817}\right) + 11 t_1 \right\}, \\
\nonumber U_{0224}(t_1, t_2) :=&\ \left\{ 2 (t_2 + t_2) + t_1 \leqslant \frac{446}{387},\ 58 (t_2 + t_2) + 9 t_1 \leqslant \frac{1264}{43},\ \frac{164}{387} \leqslant \left(t_2 + \frac{79}{817}\right), \right. \\
\nonumber & \left. \quad (t_2 + t_2) + 5 t_1 \leqslant \frac{60}{43},\ \frac{615}{43} \leqslant 29 \left(t_2 + \frac{79}{817}\right) + 10 t_1 \right\}, \\
\nonumber U_{0225}(t_1, t_2) :=&\ \left\{ \frac{27}{43} \leqslant \left(t_2 + \frac{79}{817}\right) + t_1,\ (t_2 + t_2) + t_1 \leqslant \frac{219}{301},\ \frac{41}{215} \leqslant t_1,\ 6 (t_2 + t_2) + t_1 \leqslant \frac{135}{43}, \right. \\
\nonumber & \left. \quad \frac{205}{473} \leqslant \left(t_2 + \frac{79}{817}\right),\ 7 t_1 - \left(t_2 + \frac{79}{817}\right) \leqslant \frac{55}{43},\ \frac{164}{43} \leqslant 6 \left(t_2 + \frac{79}{817}\right) + 5 t_1 \right\}, \\
\nonumber U_{0226}(t_1, t_2) :=&\ \left\{ (t_2 + t_2) + t_1 \leqslant \frac{268}{473},\ 35 (t_2 + t_2) + 23 t_1 \leqslant \frac{767}{43},\ \frac{410}{473} \leqslant 2 \left(t_2 + \frac{79}{817}\right) + t_1, \right. \\
\nonumber & \left. \quad  2 (t_2 + t_2) + 13 t_1 \leqslant \frac{130}{43},\ \frac{1476}{43} \leqslant 70 \left(t_2 + \frac{79}{817}\right) + 59 t_1 \right\}, \\
\nonumber U_{0227}(t_1, t_2) :=&\ \left\{ (t_2 + t_2) + t_1 \leqslant \frac{313}{559},\ 41 (t_2 + t_2) + 27 t_1 \leqslant \frac{902}{43},\ \frac{492}{559} \leqslant 2 \left(t_2 + \frac{79}{817}\right) + t_1, \right. \\
\nonumber & \left. \quad 2 (t_2 + t_2) + 15 t_1 \leqslant \frac{138}{43},\ \frac{1722}{43} \leqslant 82 \left(t_2 + \frac{79}{817}\right) + 69 t_1 \right\}, \\
\nonumber U_{0228}(t_1, t_2) :=&\ \left\{ 2 (t_2 + t_2) + t_1 \leqslant \frac{536}{473},\ 70 (t_2 + t_2) + 11 t_1 \leqslant \frac{1534}{43},\ \frac{205}{473} \leqslant \left(t_2 + \frac{79}{817}\right), \right. \\
\nonumber & \left. \quad (t_2 + t_2) + 6 t_1 \leqslant \frac{65}{43},\ \frac{738}{43} \leqslant 35 \left(t_2 + \frac{79}{817}\right) + 12 t_1 \right\}, \\
\nonumber U_{0229}(t_1, t_2) :=&\ \left\{ t_1 + (t_2 + t_2) \leqslant \frac{673}{1247},\ \frac{82}{473} \leqslant \left(t_2 + \frac{79}{817}\right),\ 29 t_1 - \left(t_2 + \frac{79}{817}\right) \leqslant \frac{427}{43},\ \frac{246}{817} \leqslant t_1, \right. \\
\nonumber & \left. \quad 29 (t_2 + t_2) - t_1 \leqslant \frac{263}{43},\ \frac{246}{43} \leqslant 12 t_1 + 11 \left(t_2 + \frac{79}{817}\right) \right\}, \\
\nonumber U_{0230}(t_1, t_2) :=&\ \left\{ t_1 + (t_2 + t_2) \leqslant \frac{223}{387},\ 29 t_1 + 19 (t_2 + t_2) \leqslant \frac{632}{43},\ \frac{328}{387} \leqslant 2 t_1 + \left(t_2 + \frac{79}{817}\right), \right. \\
\nonumber & \left. \quad 2 t_1 + 11 (t_2 + t_2) \leqslant \frac{120}{43},\ \frac{1230}{43} \leqslant 58 t_1 + 49 \left(t_2 + \frac{79}{817}\right) \right\}, \\
\nonumber U_{0231}(t_1, t_2) :=&\ \left\{ t_1 + (t_2 + t_2) \leqslant \frac{268}{473},\ \frac{41}{215} \leqslant \left(t_2 + \frac{79}{817}\right),\ 6 t_1 + (t_2 + t_2) \leqslant \frac{94}{43},\ \frac{82}{301} \leqslant t_1 \leqslant \frac{16}{43}, \right. \\
\nonumber & \left. \quad t_1 + 8 (t_2 + t_2) \leqslant \frac{98}{43},\ \frac{123}{43} \leqslant 6 t_1 + 5 \left(t_2 + \frac{79}{817}\right) \right\}, \\
\nonumber U_{0232}(t_1, t_2) :=&\ \left\{ t_1 + (t_2 + t_2) \leqslant \frac{571}{817},\ \frac{82}{473} \leqslant \left(t_2 + \frac{79}{817}\right),\ 12 t_1 + (t_2 + t_2) \leqslant \frac{270}{43},\ \frac{574}{1247} \leqslant t_1, \right. \\
\nonumber & \left. \quad 19 (t_2 + t_2) - t_1 \leqslant \frac{161}{43},\ \frac{328}{43} \leqslant 12 t_1 + 11 \left(t_2 + \frac{79}{817}\right) \right\}, \\
\nonumber U_{0233}(t_1, t_2) :=&\ \left\{ 2 t_1 + (t_2 + t_2) \leqslant \frac{446}{387},\ 58 t_1 + 9 (t_2 + t_2) \leqslant \frac{1264}{43},\ \frac{164}{387} \leqslant t_1, \right. \\
\nonumber & \left. \quad t_1 + 5 (t_2 + t_2) \leqslant \frac{60}{43},\ \frac{615}{43} \leqslant 29 t_1 + 10 \left(t_2 + \frac{79}{817}\right) \right\}, \\
\nonumber U_{0234}(t_1, t_2) :=&\ \left\{ \frac{27}{43} \leqslant t_1 + \left(t_2 + \frac{79}{817}\right),\ t_1 + (t_2 + t_2) \leqslant \frac{219}{301},\ \frac{41}{215} \leqslant \left(t_2 + \frac{79}{817}\right),\ 6 t_1 + (t_2 + t_2) \leqslant \frac{135}{43}, \right. \\
\nonumber & \left. \quad \frac{205}{473} \leqslant t_1,\ 7 (t_2 + t_2) - t_1 \leqslant \frac{55}{43},\ \frac{164}{43} \leqslant 6 t_1 + 5 \left(t_2 + \frac{79}{817}\right) \right\}, \\
\nonumber U_{0235}(t_1, t_2) :=&\ \left\{ t_1 + (t_2 + t_2) \leqslant \frac{268}{473},\ 35 t_1 + 23 (t_2 + t_2) \leqslant \frac{767}{43},\ \frac{410}{473} \leqslant 2 t_1 + \left(t_2 + \frac{79}{817}\right), \right. \\
\nonumber & \left. \quad  2 t_1 + 13 (t_2 + t_2) \leqslant \frac{130}{43},\ \frac{1476}{43} \leqslant 70 t_1 + 59 \left(t_2 + \frac{79}{817}\right) \right\}, \\
\nonumber U_{0236}(t_1, t_2) :=&\ \left\{ t_1 + (t_2 + t_2) \leqslant \frac{313}{559},\ 41 t_1 + 27 (t_2 + t_2) \leqslant \frac{902}{43},\ \frac{492}{559} \leqslant 2 t_1 + \left(t_2 + \frac{79}{817}\right), \right. \\
\nonumber & \left. \quad 2 t_1 + 15 (t_2 + t_2) \leqslant \frac{138}{43},\ \frac{1722}{43} \leqslant 82 t_1 + 69 \left(t_2 + \frac{79}{817}\right) \right\}, \\
\nonumber U_{0237}(t_1, t_2) :=&\ \left\{ 2 t_1 + (t_2 + t_2) \leqslant \frac{536}{473},\ 70 t_1 + 11 (t_2 + t_2) \leqslant \frac{1534}{43},\ \frac{205}{473} \leqslant t_1, \right. \\
\nonumber & \left. \quad t_1 + 6 (t_2 + t_2) \leqslant \frac{65}{43},\ \frac{738}{43} \leqslant 35 t_1 + 12 \left(t_2 + \frac{79}{817}\right) \right\}, \\
\nonumber U_{0301}(t_1, t_2) :=&\ \left\{ 2 (t_1 + t_2) + t_2 \leqslant 2 - \frac{20.5}{21.5},\ 4 (t_1 + t_2) + 6 t_2 \leqslant 4 - \frac{20.5}{21.5},\  4 t_2 \leqslant 4 - \frac{61.5}{21.5} \right\}, \\
\nonumber U_{0302}(t_1, t_2) :=&\ \left\{ (t_1 + t_2) \leqslant \frac{45}{86},\ t_2 \leqslant \frac{41}{344} \right\}, \\
\nonumber U_{0303}(t_1, t_2) :=&\ \left\{ \frac{55}{129} \leqslant (t_1 + t_2) \leqslant \frac{20}{43},\ t_2 \leqslant -\frac{1}{5}(t_1 + t_2) + \frac{12}{43} \right\}, \\
\nonumber U_{0304}(t_1, t_2) :=&\ \left\{ \frac{20}{43} \leqslant (t_1 + t_2) \leqslant \frac{41}{86},\ t_2 \leqslant \frac{1}{7}(t_1 + t_2) + \frac{55}{301} \right\}, \\
\nonumber U_{0305}(t_1, t_2) :=&\ \left\{ \frac{41}{86} \leqslant (t_1 + t_2) \leqslant \frac{227}{473},\ t_2 \leqslant -(t_1 + t_2) + \frac{219}{301} \right\}, \\
\nonumber U_{0306}(t_1, t_2) :=&\ \left\{ \frac{227}{473} \leqslant (t_1 + t_2) \leqslant \frac{2333}{4859},\ t_2 \leqslant -\frac{58}{9}(t_1 + t_2) + \frac{1264}{387} \right\}, \\
\nonumber U_{0307}(t_1, t_2) :=&\ \left\{ \frac{2333}{4859} \leqslant (t_1 + t_2) \leqslant \frac{2501}{5203},\ t_2 \leqslant -\frac{1}{6}(t_1 + t_2) + \frac{65}{258} \right\}, \\
\nonumber U_{0308}(t_1, t_2) :=&\ \left\{ \frac{2501}{5203} \leqslant (t_1 + t_2) \leqslant \frac{499}{1032},\ t_2 \leqslant -2(t_1 + t_2) + \frac{536}{473} \right\}, \\
\nonumber U_{0309}(t_1, t_2) :=&\ \left\{ \frac{499}{1032} \leqslant (t_1 + t_2) \leqslant \frac{28277}{57190},\ t_2 \leqslant -\frac{70}{11}(t_1 + t_2) + \frac{1534}{473} \right\}, \\
\nonumber U_{0310}(t_1, t_2) :=&\ \left\{ 2 (t_2 + t_2) + t_1 \leqslant 2 - \frac{20.5}{21.5},\ 4 (t_2 + t_2) + 6 t_1 \leqslant 4 - \frac{20.5}{21.5},\ 4 t_1 \leqslant 4 - \frac{61.5}{21.5}, \right\}, \\
\nonumber U_{0311}(t_1, t_2) :=&\ \left\{ (t_2 + t_2) \leqslant \frac{45}{86},\ t_1 \leqslant \frac{41}{344} \right\}, \\
\nonumber U_{0312}(t_1, t_2) :=&\ \left\{ \frac{55}{129} \leqslant \left(t_2 + \frac{79}{817}\right),\ (t_2 + t_2) \leqslant \frac{20}{43},\ t_1 \leqslant -\frac{1}{5}(t_2 + t_2) + \frac{12}{43} \right\}, \\
\nonumber U_{0313}(t_1, t_2) :=&\ \left\{ \frac{20}{43} \leqslant \left(t_2 + \frac{79}{817}\right),\ (t_2 + t_2) \leqslant \frac{41}{86},\ t_1 \leqslant \frac{1}{7}\left(t_2 + \frac{79}{817}\right) + \frac{55}{301} \right\}, \\
\nonumber U_{0314}(t_1, t_2) :=&\ \left\{ \frac{41}{86} \leqslant \left(t_2 + \frac{79}{817}\right),\ (t_2 + t_2) \leqslant \frac{227}{473},\ t_1 \leqslant -(t_2 + t_2) + \frac{219}{301} \right\}, \\
\nonumber U_{0315}(t_1, t_2) :=&\ \left\{ \frac{227}{473} \leqslant \left(t_2 + \frac{79}{817}\right),\ (t_2 + t_2) \leqslant \frac{2333}{4859},\ t_1 \leqslant -\frac{58}{9}(t_2 + t_2) + \frac{1264}{387} \right\}, \\
\nonumber U_{0316}(t_1, t_2) :=&\ \left\{ \frac{2333}{4859} \leqslant \left(t_2 + \frac{79}{817}\right),\ (t_2 + t_2) \leqslant \frac{2501}{5203},\ t_1 \leqslant -\frac{1}{6}(t_2 + t_2) + \frac{65}{258} \right\}, \\
\nonumber U_{0317}(t_1, t_2) :=&\ \left\{ \frac{2501}{5203} \leqslant \left(t_2 + \frac{79}{817}\right),\ (t_2 + t_2) \leqslant \frac{499}{1032},\ t_1 \leqslant -2(t_2 + t_2) + \frac{536}{473} \right\}, \\
\nonumber U_{0318}(t_1, t_2) :=&\ \left\{ \frac{499}{1032} \leqslant \left(t_2 + \frac{79}{817}\right),\ (t_2 + t_2) \leqslant \frac{28277}{57190},\ t_1 \leqslant -\frac{70}{11}(t_2 + t_2) + \frac{1534}{473} \right\},
\end{align}
and let $U_{0} = U_{01} \cup U_{02} \cup U_{03}$. Clearly $U_{01}$ corresponds to Lemma~\ref{l32}, $U_{02}$ corresponds to Lemmas~\ref{l33}--\ref{l36}, and $U_{03}$ corresponds to Lemmas~\ref{l43}--\ref{l44}. Then we have
\begin{align}
\nonumber S_3 =&\ \sum_{(t_1, t_2) \in U_{01}}S\left(\mathcal{A}_{p_1 p_2}, p_2\right) + \sum_{(t_1, t_2) \in U_{02}}S\left(\mathcal{A}_{p_1 p_2}, p_2\right) \\
\nonumber &+ \sum_{(t_1, t_2) \in U_{03}}S\left(\mathcal{A}_{p_1 p_2}, p_2\right) + \sum_{(t_1, t_2) \notin U_{0}}S\left(\mathcal{A}_{p_1 p_2}, p_2\right) \\
=&\ S_{301}+S_{302}+\Sigma_{303}+\Sigma_{0}.
\end{align}
Note that $U_{01}$ corresponds to Lemma~\ref{l32}, $U_{02}$ corresponds to Lemmas~\ref{l33}--\ref{l36}, and $U_{03}$ corresponds to Lemmas~\ref{l43}--\ref{l44}.

For $S_{301}$, by Lemma~\ref{l32} and Lemma~\ref{l62}, we can give an asymptotic formula. For $S_{302}$, we can apply Buchstab's identity again to get
\begin{align}
\nonumber S_{302} =&\ \sum_{(t_1, t_2) \in U_{02} }S\left(\mathcal{A}_{p_1 p_2}, p_2\right) \\
=&\ \sum_{(t_1, t_2) \in U_{02} }S\left(\mathcal{A}_{p_1 p_2}, X^{\frac{79}{817}}\right) - \sum_{\substack{(t_1, t_2) \in U_{02} \\ \frac{79}{817} \leqslant t_3 < \min\left(t_2, \frac{1}{2}(1-t_1-t_2)\right) }}S\left(\mathcal{A}_{p_1 p_2 p_3}, p_3\right).
\end{align}
By Lemma~\ref{TypeI} we have
\begin{equation}
\sum_{(t_1, t_2) \in U_{02} }S\left(\mathcal{A}_{p_1 p_2}, X^{\frac{79}{817}}\right) = \eta \sum_{(t_1, t_2) \in U_{02} }S\left(\mathcal{B}_{p_1 p_2}, X^{\frac{79}{817}}\right) + O\left(\frac{\varepsilon \eta x}{\log x}\right).
\end{equation}
By Lemmas~\ref{l33}--\ref{l36} and Lemma~\ref{l63} we have
\begin{equation}
\sum_{\substack{(t_1, t_2) \in U_{02} \\ \frac{79}{817} \leqslant t_3 < \min\left(t_2, \frac{1}{2}(1-t_1-t_2)\right) }}S\left(\mathcal{A}_{p_1 p_2 p_3}, p_3\right) = \eta \sum_{\substack{(t_1, t_2) \in U_{02} \\ \frac{79}{817} \leqslant t_3 < \min\left(t_2, \frac{1}{2}(1-t_1-t_2)\right) }}S\left(\mathcal{B}_{p_1 p_2 p_3}, p_3\right) + O\left(\frac{\varepsilon \eta x}{\log x}\right).
\end{equation}
Combining (8)--(10), we have
\begin{equation}
S_{302} = \sum_{(t_1, t_2) \in U_{02} }S\left(\mathcal{A}_{p_1 p_2}, p_2\right) = \eta \sum_{(t_1, t_2) \in U_{02} }S\left(\mathcal{B}_{p_1 p_2}, p_2\right) + O\left(\frac{\varepsilon \eta x}{\log x}\right).
\end{equation}

Before decomposing $\Sigma_{303}$, we first make a definition: we say a set of variables $(t_1, \ldots, t_n)$ is \textit{good} if the variables can be partitioned into $(t_i, h) \in U_{01}$ or $(h, t_i) \in U_{01}$. Now we can apply Buchstab's identity twice more to get
\begin{align}
\nonumber \Sigma_{303} =&\ \sum_{(t_1, t_2) \in U_{03}}S\left(\mathcal{A}_{p_1 p_2}, p_2\right) \\
\nonumber =&\ \sum_{(t_1, t_2) \in U_{03}}S\left(\mathcal{A}_{p_1 p_2}, X^{\frac{79}{817}} \right) - \sum_{\substack{(t_1, t_2) \in U_{03} \\ \frac{79}{817} \leqslant t_3 < \min\left(t_2, \frac{1}{2}(1-t_1-t_2)\right) \\ (t_1, t_2, t_3) \text{ is good} }}S\left(\mathcal{A}_{p_1 p_2 p_3}, p_3\right) \\
\nonumber & - \sum_{\substack{(t_1, t_2) \in U_{03} \\ \frac{79}{817} \leqslant t_3 < \min\left(t_2, \frac{1}{2}(1-t_1-t_2)\right) \\ (t_1, t_2, t_3) \text{ is not good} }}S\left(\mathcal{A}_{p_1 p_2 p_3}, X^{\frac{79}{817}} \right) + \sum_{\substack{(t_1, t_2) \in U_{03} \\ \frac{79}{817} \leqslant t_3 < \min\left(t_2, \frac{1}{2}(1-t_1-t_2)\right) \\ (t_1, t_2, t_3) \text{ is not good} \\ \frac{79}{817} \leqslant t_4 < \min\left(t_3, \frac{1}{2}(1-t_1-t_2-t_3)\right) \\ (t_1, t_2, t_3, t_4) \text{ is good} }}S\left(\mathcal{A}_{p_1 p_2 p_3 p_4}, p_4 \right) \\
\nonumber & + \sum_{\substack{(t_1, t_2) \in U_{03} \\ \frac{79}{817} \leqslant t_3 < \min\left(t_2, \frac{1}{2}(1-t_1-t_2)\right) \\ (t_1, t_2, t_3) \text{ is not good} \\ \frac{79}{817} \leqslant t_4 < \min\left(t_3, \frac{1}{2}(1-t_1-t_2-t_3)\right) \\ (t_1, t_2, t_3, t_4) \text{ is not good} }}S\left(\mathcal{A}_{p_1 p_2 p_3 p_4}, p_4 \right) \\
=&\ S_{3031} - S_{3032} - S_{3033} + S_{3034} + S_{3035}.
\end{align}
By Lemma~\ref{TypeI} we have
\begin{equation}
S_{3031} = \sum_{(t_1, t_2) \in U_{03}}S\left(\mathcal{A}_{p_1 p_2}, p_2\right) = \eta \sum_{(t_1, t_2) \in U_{03}}S\left(\mathcal{B}_{p_1 p_2}, p_2\right) + O\left(\frac{\varepsilon \eta x}{\log x}\right)
\end{equation}
and
\begin{equation}
S_{3033} = \sum_{\substack{(t_1, t_2) \in U_{03} \\ \frac{79}{817} \leqslant t_3 < \min\left(t_2, \frac{1}{2}(1-t_1-t_2)\right) \\ (t_1, t_2, t_3) \text{ is not good} }}S\left(\mathcal{A}_{p_1 p_2 p_3}, X^{\frac{79}{817}} \right) = \eta \sum_{\substack{(t_1, t_2) \in U_{03} \\ \frac{79}{817} \leqslant t_3 < \min\left(t_2, \frac{1}{2}(1-t_1-t_2)\right) \\ (t_1, t_2, t_3) \text{ is not good} }}S\left(\mathcal{B}_{p_1 p_2 p_3}, X^{\frac{79}{817}} \right) + O\left(\frac{\varepsilon \eta x}{\log x}\right).
\end{equation}
By Lemma~\ref{l32} and a variant of Lemma~\ref{l63} we have
\begin{equation}
S_{3032} = \sum_{\substack{(t_1, t_2) \in U_{03} \\ \frac{79}{817} \leqslant t_3 < \min\left(t_2, \frac{1}{2}(1-t_1-t_2)\right) \\ (t_1, t_2, t_3) \text{ is good} }}S\left(\mathcal{A}_{p_1 p_2 p_3}, p_3\right) = \eta \sum_{\substack{(t_1, t_2) \in U_{03} \\ \frac{79}{817} \leqslant t_3 < \min\left(t_2, \frac{1}{2}(1-t_1-t_2)\right) \\ (t_1, t_2, t_3) \text{ is good} }}S\left(\mathcal{B}_{p_1 p_2 p_3}, p_3\right) + O\left(\frac{\varepsilon \eta x}{\log x}\right)
\end{equation}
and
\begin{align}
\nonumber S_{3034} =&\ \sum_{\substack{(t_1, t_2) \in U_{03} \\ \frac{79}{817} \leqslant t_3 < \min\left(t_2, \frac{1}{2}(1-t_1-t_2)\right) \\ (t_1, t_2, t_3) \text{ is not good} \\ \frac{79}{817} \leqslant t_4 < \min\left(t_3, \frac{1}{2}(1-t_1-t_2-t_3)\right) \\ (t_1, t_2, t_3, t_4) \text{ is good} }}S\left(\mathcal{A}_{p_1 p_2 p_3 p_4}, p_4 \right) \\
=&\ \eta \sum_{\substack{(t_1, t_2) \in U_{03} \\ \frac{79}{817} \leqslant t_3 < \min\left(t_2, \frac{1}{2}(1-t_1-t_2)\right) \\ (t_1, t_2, t_3) \text{ is not good} \\ \frac{79}{817} \leqslant t_4 < \min\left(t_3, \frac{1}{2}(1-t_1-t_2-t_3)\right) \\ (t_1, t_2, t_3, t_4) \text{ is good} }}S\left(\mathcal{B}_{p_1 p_2 p_3 p_4}, p_4 \right) + O\left(\frac{\varepsilon \eta x}{\log x}\right).
\end{align}
For the remaining $S_{3035}$, we can simply discard the whole sum, leading to a loss of size
\begin{equation}
\int_{(t_1, t_2, t_3, t_4) \in V_1} \frac{\omega_{1}\left(\frac{1 - t_1 - t_2 - t_3 - t_4}{t_4}\right)}{t_1 t_2 t_3 t_4^2} d t_4 d t_3 d t_2 d t_1 < 0.03,
\end{equation}
where
\begin{align}
\nonumber V_1(t_1, t_2, t_3, t_4) :=&\ \left\{ (t_1, t_2) \in U_{03},\ \frac{79}{817} \leqslant t_3 < \min\left(t_2, \frac{1}{2}(1-t_1-t_2)\right), \right. \\
\nonumber & \quad (t_1, t_2, t_3) \text{ is not good}, \\
\nonumber & \quad \frac{79}{817} \leqslant t_4 < \min\left(t_3, \frac{1}{2}(1-t_1-t_2-t_3) \right), \\
\nonumber & \quad (t_1, t_2, t_3, t_4) \text{ is not good}, \\
\nonumber & \left. \quad \frac{79}{817} \leqslant t_1 < \frac{1}{2},\ \frac{79}{817} \leqslant t_2 < \min\left(t_1, \frac{1}{2}(1-t_1) \right) \right\}.
\end{align}
We note that further decompositions on $S_{3035}$ are possible, but we do not consider them here.

For the remaining $\Sigma_{0}$, we cannot give asymptotic formulas using Lemmas~\ref{l61}--\ref{l63}. Hence we must discard this sum, leading to a much bigger two-dimensional loss of
\begin{equation}
\int_{\frac{79}{817}}^{\frac{1}{2}} \int_{\frac{79}{817}}^{\min\left(t_1, \frac{1-t_1}{2}\right)} \mathbbm{1}_{(t_1, t_2) \notin U_{0}} \frac{\omega \left(\frac{1 - t_1 - t_2}{t_2}\right)}{t_1 t_2^2} d t_2 d t_1 < 0.992,
\end{equation}
However, we can use Buchstab's identity in reverse to get asymptotic formulas for some almost-primes counted. When $t_2 < \frac{1}{2}(1-t_1-t_2)$, we have
\begin{equation}
\sum_{\substack{(t_1, t_2) \notin U_{0} \\ t_2 < \frac{1}{2}(1-t_1-t_2)}}S\left(\mathcal{A}_{p_1 p_2}, p_2\right) = \sum_{\substack{(t_1, t_2) \notin U_{0} \\ t_2 < \frac{1}{2}(1-t_1-t_2)}}S\left(\mathcal{A}_{p_1 p_2}, \left(\frac{2X}{p_1 p_2}\right)^{\frac{1}{2}} \right) + \sum_{\substack{(t_1, t_2) \notin U_{0} \\ t_2 < t_3 < \frac{1}{2}(1-t_1-t_2)}}S\left(\mathcal{A}_{p_1 p_2 p_3}, p_3\right).
\end{equation}
This can be seen as a "reversed" application of Buchstab's identity. The first term on the right-hand side counts primes, while the second term counts almost-primes. We cannot give an asymptotic formula for the first term on the right-hand side, but we can give asymptotic formulas for part of the second term if $(t_1, t_2, t_3)$ is good. We can then subtract the size of this part from the loss of $\Sigma_{0}$. The savings from this part are larger than
\begin{equation}
\int_{(t_1, t_2, t_3) \in V_3} \frac{\omega \left(\frac{1 - t_1 - t_2 - t_3}{t_3}\right)}{t_1 t_2 t_3^2} d t_3 d t_2 d t_1 > 0.027,
\end{equation}
where
\begin{align}
\nonumber V_3(t_1, t_2, t_3) :=&\ \left\{ (t_1, t_2) \notin U_0,\ t_2 < t_3 < \frac{1}{2}(1-t_1-t_2),\ (t_1, t_2, t_3) \text{ is good}, \right. \\
\nonumber & \left. \quad \frac{79}{817} \leqslant t_1 < \frac{1}{2},\ \frac{79}{817} \leqslant t_2 < \min\left(t_1, \frac{1}{2}(1-t_1) \right) \right\}.
\end{align}

Finally, by (1)--(20) we have
$$
\pi(x+\eta x)-\pi(x) = S\left(\mathcal{A},(2 X)^{\frac{1}{2}}\right) \geqslant (1 - 0.03 - 0.992 + 0.027) \frac{\eta x}{\log x} = 0.005 \frac{\eta x}{\log x},
$$
and the proof of Theorem~\ref{t1} is completed. The exponent $\frac{1}{21.5}$ is rather near to the limit obtained by the method employed in this paper, and something like $\frac{1}{22}$ is far out of reach. With the exponent $\frac{1}{22}$, the above method only gives a lower bound constant less than $-0.1$, which is even worse than the trivial lower bound. We remark that the above method gives a lower bound constant $>0.04$ when the exponent is $\frac{1}{21}$, which is useful for an application in the next section.

\section{Applications of Theorem 1.1}
Clearly our Theorem~\ref{t1} has many interesting applications. The following application of our Theorem~\ref{t1} is about Goldbach numbers (sum of two primes) in short intervals. By combining our Theorem~\ref{t1} with the main theorem proved in \cite{BHP}, we can easily deduce the following theorem.
\begin{theorem}\label{21860}
The interval $[X, X+X^{\frac{21}{860}}]$ contains Goldbach numbers. That is,
$$
g_{n+1} - g_{n} \ll g_{n}^{21/860},
$$
where $g_{n}$ denote the $n$-th Goldbach number.
\end{theorem}
Previous exponent $\frac{21}{800}$ (see either \cite{Pintz2012} or \cite{PintzGoldbachSejtesrol}) comes from Jia's exponent $\frac{1}{20}$ \cite{Jia120}. We remark that if we focus on Maillet numbers (difference of two primes) instead of Goldbach numbers in short intervals, Pintz \cite{PintzM} improved this exponent to any $\varepsilon > 0$.

Another application of our Theorem~\ref{t1} is about prime values of the integer parts of real sequences, improving the previous result of Harman \cite{Harman2001} by adding one more term on both of the sequences $[p^k \alpha]$ and $[(p \alpha)^k]$.
\begin{theorem}\label{t2}
For almost all $\alpha>0$, both of the following statements are true:

(1) Infinitely often $p, [p \alpha], [p^2 \alpha], \ldots, [p^{21} \alpha]$ are all prime.

(2) Infinitely often $p, [p \alpha], [(p \alpha)^2], \ldots, [(p \alpha)^{21}]$ are all prime.
\end{theorem}
The proof of Theorem~\ref{t2} can be done by replacing [\cite{Harman2001}, Lemma 4] by a variant of our Theorem~\ref{t1}. In this way, the ratio $\frac{20}{19}$ in [\cite{Harman2001}, Theorem 4] can be improved to $\frac{21.5}{20.5}$. However, due to the lack of a positive lower bound when the exponent is $\frac{1}{22}$, we cannot add even one more term on our Theorem~\ref{t2}.

The last application of our Theorem~\ref{t1} is about the Three Primes Theorem with small prime solutions, improving the previous result of Cai \cite{Cai2013} by reducing the exponent $\frac{11}{400}$ to $\frac{11}{420}$.
\begin{theorem}\label{t3}
Let $Y$ denote a sufficiently large odd integer. The equation
$$
Y = p_1 + p_2 + p_3, \quad p_1 \leqslant Y^{\frac{11}{420}}
$$
is solvable.
\end{theorem}
The exponent $\frac{11}{420}$ comes from a weaker version of our Theorem~\ref{t1} with exponent $\frac{1}{21}$. Note that we have $u_0 > 0.04$ in this case, and we can show that $u_1 < 2.88$ by similar arguments as in \cite{Cai2013}. By using the vector sieve together with the same bounds for $v_0$ and $v_1$ as in \cite{Cai2013}, the proof of our Theorem~\ref{t3} is essentially the same as the proof of [\cite{Cai2013}, Theorem]. However, our sieve machinery is not very numerically significant when the exponent drops to $\frac{1}{21.5}$ (that is, we only have $u_0 > 0.005$ by our Theorem~\ref{t1}), we cannot use vector sieve to prove Theorem~\ref{t3} with condition $p_1 \leqslant Y^{\frac{11}{430}}$.


\newpage

\bibliographystyle{plain}
\bibliography{bib}

\end{document}